\documentclass[a4paper]{amsart}
\usepackage{amssymb}
\usepackage{amscd}
\usepackage{amsthm}
\usepackage{amsmath}
\usepackage{latexsym}
\usepackage[all]{xy}

\theoremstyle{plain}
\newtheorem{theorem}{Theorem}
\newtheorem{corollary}[theorem]{Corollary}
\newtheorem{lemma}[theorem]{Lemma}
\newtheorem{proposition}[theorem]{Proposition}

\theoremstyle{definition}\newtheorem{definition}[theorem]{Definition}
\newtheorem{example}[theorem]{Example}

\theoremstyle{remark}

\newtheorem{remark}[theorem]{Remark}
\numberwithin{theorem}{section}

\theoremstyle{plain}
\newtheorem*{conjHochster}{Hochster's Conjecture}
\newtheorem*{conjSerre}{Serre's Conjectures}

\newenvironment{ii}
{ \begin{enumerate}}
{\end{enumerate}}

\newcommand{\Spec}[1]{\mathrm{Spec}(#1)}

\newcommand{\Add}{\mbox{\rm{Add\,}}}
\newcommand{\Prod}{\mbox{\rm{Prod\,}}}

\newcommand{\Gen}{\mbox{\rm{Gen\,}}}

\newcommand{\Cog}{\mbox{\rm{Cog\,}}}

\newcommand{\Ass}[1]{\mathrm{Ass} \,#1}
\newcommand{\Coass}[1]{\mathrm{Coass} \,#1}

\newcommand{\Img}{\mbox{\rm{Im\,}}}

\newcommand{\Hom}[3]{\mbox{\rm{Hom}}_{#1}(#2,#3)}
\newcommand{\Ext}[4]{\mbox{\rm{Ext}}^{#1}_{#2}(#3,#4)}
\newcommand{\Tor}[4]{\mbox{\rm{Tor}}_{#1}^{#2}(#3,#4)}

\newcommand{\rfmod}[1]{\mbox{\rm{mod}--}{#1}}
\newcommand{\rmod}[1]{\mbox{\rm{Mod}--}{#1}}
\newcommand{\lfmod}[1]{{#1}\mbox{--\rm{mod}}}
\newcommand{\lmod}[1]{{#1}\mbox{--\rm{Mod}}}
\newcommand{\ModR}{\rmod R}

\newcommand{\pd}[2]{\mathrm{proj.dim}_{#1}{#2}}
\newcommand{\fd}[2]{\mathrm{flat.dim}_{#1}{#2}}
\newcommand{\id}[2]{\mathrm{inj.dim}_{#1}{#2}}
\newcommand{\dt}[1]{\mathrm{depth}\,{#1}}
\newcommand{\Kd}[1]{\mathrm{Kdim}\,{#1}}
\newcommand{\htt}[1]{\mathrm{ht}\;{#1}}

\newcommand{\Ann}[1]{\mbox{\rm{Ann}}(#1)}
\newcommand{\Supp}[1]{\mbox{\rm{Supp}} \,#1}

\newcommand{\Syz}[2]{\Omega^{#1}(#2)}
\newcommand{\Cosyz}[2]{\mho_{#1}(#2)}
\newcommand{\Tr}[1]{\mathrm{\rm Tr}({#1})}
\newcommand{\card}[1]{\left\lvert{#1}\right\rvert}

\newcommand{\p}{\mathfrak{p}}
\newcommand{\q}{\mathfrak{q}}
\newcommand{\m}{\mathfrak{m}}

\newcommand{\Z}{\mathbb{Z}}

\newcommand{\A}{\mathcal{A}}

\newcommand{\C}{\mathcal{C}}
\newcommand{\F}{\mathcal{F}}
\newcommand{\G}{\mathcal{G}}
\newcommand{\clS}{\mathcal{S}}
\newcommand{\T}{\mathcal{T}}
\newcommand{\D}{\mathcal{D}}

\newcommand{\la}{\longrightarrow}

\renewcommand{\iff}{if and only if }
\newcommand{\st}{such that }
\newcommand{\wrt}{with respect to }

\begin{document}

\title{Tilting, cotilting, and spectra of commutative noetherian rings}

\author{Lidia Angeleri H\" ugel}
\address{Dipartimento di Informatica -
Settore di Matematica\\ 
Universit\`a degli Studi di Verona\\
Strada le Grazie 15 - Ca' Vignal\\
37134 Verona, Italy}
\email{lidia.angeleri@univr.it}

\author{David Posp{\'\i}{\v s}il} 
\author{Jan \v{S}\v{t}ov\'\i\v{c}ek} 
\author{Jan Trlifaj}
\address{Charles University, Faculty of Mathematics and Physics, Department of Algebra \\
Sokolovsk\'{a} 83, 186 75 Prague 8, Czech Republic}
\email{dpos@karlin.mff.cuni.cz}
\email{stovicek@karlin.mff.cuni.cz}
\email{trlifaj@karlin.mff.cuni.cz}

\date{\today}
\thanks{Research supported by GA\v{C}R 201/09/0816, GA\v{C}R 201/09/H012, GA\v{C}R P201/10/P084, as well as by  MEC-DGESIC (Spain) through Project MTM2008-06201-C02-01, and by the Comissionat Per Universitats i Recerca de la Generalitat de Catalunya through Project 2009 SGR 1389.}

\begin{abstract}
We classify all tilting and cotilting classes over commutative noetherian rings in terms of descending sequences of specialization closed subsets of the Zariski spectrum. Consequently, all resolving subcategories of finitely generated modules of bounded projective dimension are classified. We also relate our results to Hochster's conjecture on the existence of finitely generated maximal Cohen-Macaulay modules.
\end{abstract}

\maketitle

\section*{Introduction}

It is well known that the Zariski spectrum of a commutative noetherian ring $R$ can be used to classify various structures over $R$. For example, it was shown by Gabriel in 1962 that the hereditary torsion pairs in the module category $\rmod R$ are  parametrized by the subsets of $\Spec{R}$ that are closed under specialization. An analogous result holds true at the level of the derived category: based on work of Hopkins, a one-to-one correspondence between the specialization closed subsets of $\Spec{R}$ and the smashing subcategories of the unbounded derived category $\mathcal D(R)$ was   established by  Neeman in 1992.

In the present paper, we restrict to specialization closed subsets of $\Spec{R}$ that do not contain associated primes of $R$, and show that they parametrize all $1$-cotilting classes of $R$-modules. We then use this approach to give for each $n \geq 1$ a complete classification of $n$-tilting and $n$-cotilting classes in $\rmod R$ 
in terms of finite sequences of subsets of the Zariski spectrum of $R$ (see Theorem~\ref{thm:class-main} below).  

While classification results of this kind are usually proved by considering the tilting setting first and then passing to the cotilting one by a sort of duality, the approach applied here is the very opposite. The key point rests in an analysis of the associated primes of cotilting classes and their cosyzygy classes. The classification of the tilting classes comes a posteriori, by employing the Auslander-Bridger transpose. 
For $n = 1$, we prove an additional result: In Theorem~\ref{thm:cofinite type}, we show that all $1$-cotilting modules over one-sided noetherian rings are of cofinite type, that is, equivalent to duals of $1$-tilting modules. 

We also prove several results for tilting and cotilting classes in the setting of commutative noetherian rings which fail for general rings:

\begin{ii}
\item For each $n \geq 1$, the elementary duality  
gives a bijection between $n$-tilting and $n$-cotilting classes of modules. (For general rings, there are more 1-cotilting classes than duals of 1-tilting classes: Bazzoni constructed such examples for certain commutative non-noetherian rings in~\cite{B}.)  

\item All $n$-cotilting classes are closed under taking injective envelopes by Proposition~\ref{prop:k(p)_inside}(ii).
In particular, $1$-cotilting classes are precisely the torsion-free classes of faithful hereditary torsion pairs (Theorem~\ref{thm:class-1-cot}). (Note that $1$-cotilting classes over general rings need not be closed under injective envelopes; see~\cite[Theorem 2.5]{D}.)

\item Up to adding an injective direct summand, a minimal cosyzygy of an $n$-cotilting module is $(n-1)$-cotilting (Corollary~\ref{cor:module_shifting}). (Again, this typically fails for non-commutative rings, even for finite dimensional algebras over a field, since the cosyzygy often has self-extensions.)
\end{ii}

\smallskip

Although the tilting and cotilting modules over commutative rings are inherently infinitely generated in all non-trivial cases, our results have consequences for finitely generated modules as well.

First, as a side result we classify all resolving subcategories of finitely generated modules of bounded projective dimension in Corollary~\ref{cor:parres} 
\footnote{Added in proof: An alternative description of resolving subcategories of finitely generated modules of bounded projective dimension in terms of grade consistent functions on $\Spec{R}$ has recently been obtained by Dao and Takahashi \cite{DT}.} and prove that they hardly ever provide for approximations. 

Secondly, we relate our results to a conjecture due to Hochster claiming the existence of finitely generated maximal Cohen-Macaulay $R/\p$-modules for regular local rings $R$ and give information about the structure of these hypothetical modules in Theorem~\ref{thm:Lp_to_Kp}.
 
\section{Preliminaries}
\label{sec:prelim}

\subsection{Basic notations}
\label{subsec:basics}

For a ring $R$, we denote by $\rmod R$ the category of all (unitary right $R$-) modules, and by $\rfmod R$ its subcategory consisting of all finitely generated modules. Similarly, we define $\lmod R$ and $\lfmod R$ using left $R$-modules.

For a module $M$, $\Add M$ denotes the class of all direct summands of (possibly infinite) direct sums of copies of the module $M$. Similarly, $\Prod M$ denotes the class of all direct summands of direct products of copies of $M$. Further, we denote by $\Syz{}M$ a syzygy of $M$ and by $\Cosyz{}M$ a minimal cosyzygy of $M$. That is, $\Cosyz{}M = E(M)/M$, where $E(M)$ is an injective envelope of $M$. As usual, we define also higher cosyzygies: Given a module $M$,
\[ 0 \la M \la E_0(M) \la E_1(M) \la E_2(M) \la \cdots \]
will stand for the minimal injective coresolution and the image of $E_{i-1}(M) \to E_i(M)$ for $i \geq 1$ will be denoted by
$\Cosyz iM$. That is, $\Cosyz{}M = \Cosyz 1M$. We refrain from the usual notation $\Omega^{-i}(M)$ for the $i$-th cosyzygy for we require the following convention:
\[ \Cosyz 0M = M \qquad \textrm{and} \qquad \Cosyz iM = 0 \textrm{ for all } i<0. \]
Thus, we need to distinguish between syzygies and negative cosyzygies.

Given a class $\clS$ of right modules, we denote:
\begin{align*}
\clS^\perp &= \{ M \in \rmod R \mid \Ext iRSM = 0 \textrm{ for all } S \in \clS \textrm{ and } i\geq 1 \},    \\
{^\perp \clS} &= \{ M \in \rmod R \mid \Ext iRMS = 0 \textrm{ for all } S \in \clS \textrm{ and } i\geq 1 \}.
\end{align*}
If $\clS = \{S\}$ is a singleton, we shorten the notation to $S^\perp$ and ${^\perp S}$. A similar notation is used for the classes of modules orthogonal \wrt the Tor functor:
\[ \clS^\intercal = \{ M \in \lmod R \mid \Tor iRSM = 0 \textrm{ for all } S \in \clS \textrm{ and } i\geq 1 \}. \]

Given a class $\clS \subseteq \rmod R$ and a module $M$, a well-ordered chain of submodules
\[ 0 = M_0 \subseteq M_1 \subseteq M_2 \subseteq \cdots \subseteq M_\alpha \subseteq M_{\alpha+1} \subseteq \cdots M_\sigma = M, \]
is called an \emph{$\clS$-filtration} of $M$ if $M_\beta = \bigcup_{\alpha<\beta} M_\alpha$ for every limit ordinal $\beta\le\sigma$ and up to isomorphism $M_{\alpha+1}/M_\alpha \in \clS$ for each $\alpha<\sigma$. A module is called \emph{$\clS$-filtered} if it has at least one $\clS$-filtration.

Further, given an abelian category $\A$ (in our case typically $\A = \rmod R$, or $\A = \rfmod R$ if $R$ is right noetherian), a pair of full subcategories $(\T,\F)$ is called a \emph{torsion pair} if
\begin{ii}
  \item $\Hom{\A}TF = 0$ for each $T \in \T$ and $F \in \F$;
  \item For each $M \in \A$ there is an exact sequence $0 \to T \to M \to F \to 0$ with $T \in \T$ and $F \in \F$.
\end{ii}
In such a case, $\T$ is called a \emph{torsion class} and $\F$ a \emph{torsion-free class}. A standard and easy but useful observation is the following:

\begin{lemma} \label{lem:comparing}
Let $(\T,\F)$ and $(\T',\F')$ be torsion pairs in an abelian category. If $\T' \subseteq \T$ and $\F' \subseteq \F$, then $\T = \T'$ and $\F = \F'$.
\end{lemma}

If $\A = \rmod R$, it is well-known that $\F$ is the torsion-free class of a torsion pair \iff $\F$ is closed under submodules, extensions and direct products. Similarly, torsion classes are precisely those closed under factor modules, extensions and direct sums. For $\A = \rfmod R$ and $R$ right noetherian, any torsion-free class $\F$ is closed under submodules and extensions (so also under finite products), but some caution is due here as these closure properties do not characterize torsion-free classes. Consider for instance $R = \Z$ and the class $\F$ of all finite abelian groups.

Let us conclude this discussion with two more properties which torsion pairs in $\rmod R$ can possess.

\begin{definition} \label{def:tor_pair}
Let $(\T,\F)$ be a torsion pair in $\rmod R$. Then $(\T,\F)$ is \emph{hereditary} if $\T$ is closed under submodules, or equivalently by~\cite[Chapter VI, Proposition 3.2]{St}, if $\F$ is closed under taking injective envelopes. The torsion pair is called \emph{faithful} if $R \in \F$.
\end{definition}

\subsection{Commutative algebra essentials}
\label{subsec:comm_alg}

For a commutative noetherian ring $R$, we denote by $\Spec R$ the spectrum of $R$. The spectrum is well-known to carry the Zariski topology, where the closed sets are those of the form
\[ V(I) = \{ \p\in\Spec R \mid \p \supseteq I \}, \]
for some subset $I \subseteq R$. If $I = \{f\}$ is a singleton, we again write just $V(f)$.

Given $M \in \rmod R$, $\Ass M$ denotes the set of all associated primes of $M$, and $\Supp M$ the support of $M$. For $\mathcal C \subseteq \rmod R$, we let
\[
\Ass {\mathcal C} = \bigcup_{M \in \mathcal C} \Ass M \quad
\textrm{ and }
\quad \Supp {\mathcal C} = \bigcup_{M \in \mathcal C} \Supp M.
\]
For $\p \in \Spec R$, we denote by $R_\p$ the localization of $R$ at $\p$, and by $k(\p)  = R_\p/\p_\p$ the residue field.

If $M \in \rmod R$, $\p \in \Spec R$ and $i \ge 0$, the \emph{Bass invariant} $\mu_i(\p,M)$ is defined as the number of direct summands isomorphic to $E(R/\p)$ in a decomposition of $E_i(M)$ into indecomposable direct summands (see e.g.~\cite[\S 9.2]{EJ} or~\cite[\S 3.2]{BrH}). That is,
\[ E_i(M) = \bigoplus_{\p \in \Spec R} E(R/\p)^{(\mu_i(\p,M))}. \]

The relation of associated primes to these invariants is captured by the following lemma due to Bass:

\begin{lemma} \label{lem:Bass}
Let $M$ be an $R$-module, $\p \in \Spec R$ and $i \ge 0$. Then
\[ \mu_i(\p,M) = \dim_{k(\p)} \Ext{i}{R_\p}{k(\p)}{M_\p}, \]
and we have the following equivalences:
\[
\p \in \Ass{\Cosyz iM}
\quad \Longleftrightarrow \quad
\p \in \Ass{E_i(M)}
\quad \Longleftrightarrow \quad
\mu_i(\p,M) \ne 0.
\]
\end{lemma}

\begin{proof}
For the equality above we refer for instance to~\cite[Proposition 3.2.9]{BrH} or~\cite[Theorem 9.2.4]{EJ}. The first equivalence below is proved in~\cite[Lemma 3.2.7]{BrH}. For the second, we use the equality $\mu_i(\p,M) = \dim_{k(\p)} \Hom{R_\p}{k(\p)}{E_i(M_\p)}$ from the proof of \cite[Proposition 3.2.9]{BrH} or~\cite[Theorem 9.2.4]{EJ}.
\end{proof}

As a consequence, we can extend classic relations between associated prime ideals of the terms of a short exact sequence to their cosyzygies: 

\begin{lemma} \label{lem:ses}
Let $0 \to K \to L \to M \to 0$ be a short exact sequence of $R$-modules and $i \in \Z$. Then the following hold:
\begin{ii}
 \item $\Ass{\Cosyz iK} \subseteq \Ass{\Cosyz{i-1}M} \cup \Ass{\Cosyz iL}$.
 \item $\Ass{\Cosyz iL} \;\subseteq \Ass{\Cosyz iK} \cup \Ass{\Cosyz iM}$.
 \item $\Ass{\Cosyz iM} \subseteq \Ass{\Cosyz iL} \cup \Ass{\Cosyz{i+1}K}$.
\end{ii}
\end{lemma}

\begin{proof}
Given any $\p \in \Spec R$, we consider the long exact sequence of Hom and Ext groups, which we obtain by applying the functor $\Hom{R_\p}{k(\p)}-$ on the localized short exact sequence
\[ 0 \la K_\p \la L_\p \la M_\p \la 0. \]
The lemma is then an easy consequence of Lemma~\ref{lem:Bass}.
\end{proof}

In particular, we obtain information on associated primes of syzygy modules.

\begin{corollary} \label{cor:syzygies}
Let $M$ be an $R$-module, $\ell \ge 1$ and $K$ be an $\ell$-th syzygy of $M$. Then for any $i \in \Z$ we have:
\[ \Ass{\Cosyz iK} \subseteq \Ass{\Cosyz{i-\ell}M} \cup \bigcup_{j=0}^{\ell-1} \Ass{\Cosyz{i-j}R}, \]
and 
\[ \Ass{M} \subseteq  \bigcup_{j=0}^{\ell-1} \Ass{\Cosyz{j}R}\cup \Ass{\Cosyz{\ell}K}.  \]
\end{corollary}

\begin{remark} \label{rem:syzygy_indep}
We stress that according to our convention, $\Cosyz{i-\ell}M = 0$ for $i-\ell < 0$. Thus, the right-hand term does not depend on $M$ for $i < \ell$.
\end{remark}

\begin{proof}
This is easily obtained from Lemma~\ref{lem:ses}(i) by induction on $\ell$. We also use that $\Ass{\Cosyz jP} \subseteq \Ass{\Cosyz jR}$ for any $j \in \Z$ and any projective module $P$.
\end{proof}

We finish by recalling a well-known property of the residue field considered as $R$-module (see e.g.\ \cite[Theorem 18.4]{MATSUMURA}), and its consequences:

\begin{lemma}\label{lem:kp}
Let $\p \in \Spec R$. Then $E(R/\p) \cong E_{R_{\p}}(k(\p))$ as $R$-modules. In particular:
\begin{enumerate}
 \item $E(R/\p)$ is $\{ k(\p) \}$-filtered and $\Ass{k(\p)} = \Ass{E(R/\p)} = \{\p\}$;
 \item $\Cosyz{i}{k(\p)}$ is $\{ k(\p) \}$-filtered and $\Ass{\Cosyz{i}{k(\p)}} \subseteq \{\p\}$ for each $i \ge 1$.
\end{enumerate}
\end{lemma} 

\subsection{Tilting and cotilting modules and classes}
\label{subsec:tilting}

Next, we recall the notion of an (infinitely generated) tilting module from \cite{CT,AC}:

\begin{definition} \label{def:tilt}
Let $R$ be a ring. A module $T$ is \emph{tilting} provided that
\begin{itemize}
\item[\rm{(T1)}] $T$ has finite projective dimension.
\item[\rm{(T2)}] $\Ext iRT{T^{(\kappa)}} = 0$ for all $i \ge 1$ and all cardinals $\kappa$.
\item[\rm{(T3)}] There is an exact sequence $0 \to R \to T_0  \to T_ 1 \to \dots \to T_r \to 0$ where $T_0, T_1, \dots, T_r \in \Add T$.
\end{itemize}

\noindent The class $T^\perp = \{ M \in \rmod R \mid \Ext iRTM = 0 \textrm{ for each } i \ge 1 \}$ is called the \emph{tilting class} induced by $T$. Given an integer $n \ge 0$, a tilting module as well as its associated class are called \emph{$n$-tilting} provided the projective dimension of $T$ is at most $n$. We recall that in such a case we can chose the sequence in (T3) so that $r \le n$ (see~\cite[Proposition 3.5]{BAZ}).

\noindent If $T$ and $T^\prime$ are tilting modules, then $T$ is said to be \emph{equivalent} to $T^\prime$ provided that $T^{\perp} = (T^\prime)^{\perp}$, or equivalently by~\cite[Lemma 5.1.12]{GT}, $T^\prime \in \Add T$.
\end{definition}

The structure of tilting modules over commutative noetherian rings is rather different from the classic case of artin algebras. The key point is the absence of non-trivial finitely generated tilting modules:  

\begin{lemma}\label{lem:nofin} \cite{CM,PT}
Let $R$ be a commutative noetherian ring and $T$ be a finitely generated module. Then $T$ is tilting, if and only if $T$ is projective.
\end{lemma}

Even though the tilting module $T$ is infinitely generated, the tilting class $T^\perp$ is always determined by a set $\mathcal S$ of finitely generated modules of bounded projective dimension. This was proved in~\cite{BaSt}, based on the corresponding result~\cite{BH} for $1$-tilting modules. We will call a subclass $\clS$ of $\rfmod R$ \emph{resolving} in case $\mathcal S$ is closed under extensions, direct summands, kernels of epimorphisms, and $R \in \mathcal S$. If $\clS$ consists of modules of projective dimension $\le 1$, the requirement of $\clS$ being closed under kernels of epimorphisms is redundant by~\cite[Lemma 5.2.22]{GT}. Using results from \cite{AHT,BH,BaSt}, we learn that resolving subclasses of $\rfmod R$ parametrize tilting classes (and hence also the tilting modules up to equivalence):

\begin{lemma}\label{lem:resolv} \cite[5.2.23]{GT}
Let $R$ be a right noetherian ring and $n \ge 0$. Then there is a bijective correspondence between
\begin{ii}
  \item $n$-tilting classes $\T$ in $\rmod R$, and
  \item resolving subclasses $\mathcal S$ of $\rfmod R$ consisting of modules of projective dimension $\leq n$.
\end{ii}
The correspondence is given by the assignments $\T \mapsto {^\perp \T} \cap \rfmod R$ and $\clS \mapsto \clS ^\perp$.
\end{lemma}

The dual notions of a cotilting module and a cotilting class are defined as follows: 

\begin{definition}\label{def:cotilt}
Let $R$ be a ring. A module $C$ is \emph{cotilting} provided that
\begin{itemize}
\item[\rm{(C1)}] $C$ has finite injective dimension.
\item[\rm{(C2)}] $\Ext iR{C^{\kappa}}C = 0$ for all $i \geq 1$ and all cardinals $\kappa$.
\item[\rm{(C3)}] There is an exact sequence $0 \to C_r \to \dots \to C_1 \to C_0 \to W \to 0$ where $W$ is an injective cogenerator of $\rmod R$ and $C_0, C_1, \dots C_r \in \Prod C$.
\end{itemize}

\noindent The class ${^\perp C} = \{ M \in \rmod R \mid \Ext iRMC = 0 \textrm{ for all } i \geq 1 \}$ is the \emph{cotilting class} induced by $C$. Again, if the injective dimension of $C$ is at most $n$, we call $C$ and ${^\perp C}$ an \emph{$n$-cotilting} module and class, respectively.

\noindent If $C$ and $C^\prime$ are cotilting modules, then $C$ is said to be \emph{equivalent} to $C^\prime$ provided that $^{\perp} C = {}^\perp C^\prime$, or equivalently by~\cite[Remark 8.1.6]{GT}, $C^\prime \in \Prod C$.
\end{definition}

If $T$ is an $n$-tilting right $R$-module, then the character module
\[ C = T^+ = \Hom {\mathbb Z}T{\mathbb Q/\mathbb Z} \]
is an $n$-cotilting left $R$-module; see~\cite[Proposition 2.3]{AHT}. By Lemma~\ref{lem:resolv}, the induced tilting class $\T = T^\perp$ equals $\clS ^\perp$ where $\clS = {^\perp \T} \cap \rfmod R$ is a resolving subclass of $\rfmod R$. The cotilting class $\C$ induced by $C$ in $\lmod R$ is then easily seen to be 
\[ \C = {}^\perp C = T^\intercal = \clS ^\intercal = \{ M \in \lmod R \mid \Tor 1RSM \hbox{ for all } S \in \clS \}. \]
We will call $\C$ the cotilting class {\em associated} to the tilting class $\mathcal T$.

It follows that that tilting modules $T$ and $T'$ are equivalent, \iff the character modules $T^+$ and $(T^\prime)^+$ are equivalent as cotilting left $R$-modules; see~\cite[Theorem 8.1.13]{GT}. Therefore, the assignment $T \mapsto T^+$ induces an injective map from equivalence classes of tilting to equivalence classes of cotilting modules. For $R$ 
noetherian, this map, as we will show, is a bijection, but for non-noetherian commutative rings the surjectivity may fail; see~\cite{B}. Let us summarize the properties we need.

\begin{lemma} \label{lem:cofinite}
Let $R$ be right noetherian ring and $n \ge 0$. Then the following holds:
\begin{ii}
\item If $\clS \subseteq \rfmod R$ is a class of finitely generated modules of projective dimension bounded by $n$, then $\clS^\perp$ is an $n$-tilting class in $\rmod R$ and $\clS^\intercal$ is the associated $n$-cotilting class in $\lmod R$.
\item An $n$-cotilting class $\C$ in $\lmod R$ is associated to a tilting class \iff there exists a class $\clS$ of finitely generated modules of projective dimension $\leq n$ \st $\C = \clS^\intercal$.
\end{ii}
\end{lemma}

\begin{proof}
For (i), $\clS^\perp$ is an $n$-tilting class by~\cite[Theorem 5.2.2]{GT} and $\clS^\intercal$ is $n$-cotilting by~\cite[Theorem 8.1.12]{GT}. The cotilting class $\clS^\intercal$ is associated to the tilting class $\clS^\perp$ by~\cite[Theorem 8.1.2]{GT}. Part (ii) is proved in~\cite[Theorem 8.1.13(a)]{GT}.
\end{proof}

\begin{remark} \label{rem:elem_dual}
The relation between a tilting class $\T$ and the associated cotilting class $\C$ can be interpreted using model-theoretic means in terms of the so-called elementary duality. Namely, $\T$ and $\C$ can be axiomatized in the first order language of the right (left, resp.) $R$-modules (cf.~\cite[5.2.2 and 8.1.7]{GT}) and the corresponding theories are given by mutually dual primitive positive formulas. We refer to~\cite[Section 1.3]{PREST} for more details and references on the model-theoretic background.
\end{remark}

\section{The one-dimensional case}
\label{sec:dim_one}

From this point on, unless explicitly specified otherwise, we will assume that our base ring $R$ is commutative and noetherian. 

We will treat separately the case of $1$-tilting and $1$-cotilting modules. We have chosen such presentation for two reasons. First, the arguments for this special situation are simpler and more transparent. Second, the one-dimensional case is tightly connected to the classical notion of Gabriel topology and the abelian quotients of the category $\rmod R$. We refer to~\cite{St} for details on the latter concepts.

To start with, we recall~\cite[Lemma 6.1.2]{GT}: $T \in \rmod R$ is $1$-tilting \iff $T^\perp = \Gen (T)$ where the latter denotes the class of all modules generated by $T$. In particular, $T^\perp$ is a torsion class in $\rmod R$. Dually by~\cite[Lemma 8.2.2]{GT}, a module $C$ is $1$-cotilting \iff ${^\perp C} = \Cog (C)$ where the latter denotes the class of all modules cogenerated by $C$. Thus, $^\perp C$ is a torsion free class.

Our aim is to show that a torsion pair in $\rmod R$ is of the form $(\T,\Cog(C))$ for a $1$-cotilting module $C$ \iff it is faithful and hereditary. Moreover, we are going to classify such torsion pairs in terms of certain subsets of $\Spec R$. To this end, we introduce the following terminology:

\begin{definition}\label{def:closure}
For any subset $X \subseteq \Spec{R}$ we say that $X$ is \emph{closed under generalization} (under \emph{specialization}, resp.) if for any $\p \in X$ and any $\q \in \Spec{R}$ we have $\q \in X$ whenever $\q \subseteq \p$ ($\q \supseteq \p$, resp.). In other words, $X$ is a lower (upper, resp.) set in the poset $(\Spec R,\subseteq )$.
\end{definition}

\begin{remark}\label{rem:generalization-vs-generization}
In the algebraic geometry literature, one often uses the term \emph{closed} (or \emph{stable}) \emph{under generization} instead of generalization; see e.g.~\cite[Exercise 3.17(e), p. 94]{Har} or~\cite[Definition 2.8]{GW}.
\end{remark}

Further, we recall that Gabriel established a one-to-one correspondence between the subsets of $\Spec{R}$ closed under specialization and certain linear topologies on~$R$. On the other hand, there is a bijective correspondence between these Gabriel topologies and hereditary torsion pairs in $\rmod R$. Let us look closer at this relationship.

\begin{proposition}\label{prop:Gabriel}
Every subset $Y \subseteq \Spec{R}$ closed under specialization gives rise to a Gabriel topology on~$R$ (in the sense of~\cite[\S VI.5]{St}), given by the following set of open neighbourhoods of $0 \in R$, where all the $I$ are ideals:
\[ \G_Y= \{ I\subseteq R \mid V(I) \subseteq Y\}. \]
Then $\mathcal G_Y \cap \Spec R = Y$ and the set $Y$ also determines a hereditary torsion pair $(\mathcal \T(Y),\mathcal \F(Y))$, where:
\begin{align*}
\T(Y) &= \{ M\in \ModR \mid \Supp{M}\subseteq Y\},       \\
\F(Y) &= \{ M\in \ModR \mid \Ass{M}\cap Y=\emptyset\}.
\end{align*}
We further have the following:
\begin{ii}
\item The assignments $Y \mapsto \G_Y$ and $Y \mapsto (\T(Y),\F(Y))$ define bijective correspondences between the subsets of $\Spec R$ closed under specialization, the Gabriel topologies on $R$, and the hereditary torsion pairs in $\ModR$.
\item $\T(Y)=\{ M\in \ModR \mid \Hom{R}{M}{E(R/\q)}=0 \textrm{ for all } \q\notin Y\}$ and $\T(Y)$ contains all $E(R/\p)$ with $\p\in Y$.
\item $\F(Y)=\{ M\in \ModR \mid \Hom{R}{R/\p}{M}   =0 \textrm{ for all } \p\in Y\}$ and $\F(Y)$ contains all $E(R/\q)$ with $\q\notin Y$.
\item $(\T(Y),\F(Y))$ is a torsion theory of finite type, that is,
\[ \T(Y) = \varinjlim (\T(Y) \cap \rfmod R) \textrm{ and } \F(Y)=\varinjlim (\F(Y)\cap\rfmod R). \]
\end{ii}
\end{proposition}

\begin{proof}
First of all, observe that $\mathcal G_Y\cap\Spec{R}=Y$ as $Y$ is closed under specialization. For the fact that $\G_Y$ is a Gabriel topology we refer to \cite[Theorem VI.5.1 and \S VI.6.6]{St}. Next, $\T(Y)$ defined as above is clearly closed under submodules, factor modules, extensions and direct sums, so it is a torsion class in a hereditary torsion pair. We claim that $\F(Y)$ is the corresponding torsion-free class. Indeed, given $M \in \F(Y)$, denote by $t(M)$ the $\T(Y)$-torsion part of $M$. Then
\[ \Ass{t(M)} \subseteq \Ass{M} \cap \Ass{\T(Y)} \subseteq \Ass{M} \cap Y = \emptyset. \]
Hence $t(M) = 0$ by~\cite[2.4.3]{EJ} and $M$ is torsion-free. Conversely, if $M$ is torsion-free, we must have $\Ass{M} \cap Y = \emptyset$. This is since for any $\p \in \Ass{M}$ we have an embedding $R/\p \hookrightarrow M$, but if $\p \in Y$, we have $R/\p \in \T(Y)$ owing to the fact that $Y$ is closed under specialization and $\Supp{R/\p} = V(\p) \subseteq Y$. This proves the claim, showing that the latter correspondence is well-defined.

For statement (i), note that the inverse of $Y \mapsto \G_Y$ is given by the assignment $\G \mapsto \G \cap \Spec{R}$, where $\G$ is a Gabriel topology. This follows from the equality $\mathcal G_Y\cap\Spec{R}=Y$ and~\cite[VI.6.13 and VI.6.15]{St}. It is well-known that Gabriel topologies are in bijection with hereditary torsion pairs; the hereditary torsion pair $(\T'(Y),\F'(Y))$ corresponding to $\mathcal G_Y$ is given by
\[ \T'(Y) = \{M\in \ModR\mid \Ann{x} \in \G_Y \textrm{ for all } x\in M\}, \]
see~\cite[Theorem VI.5.1]{St}. Equivalently,
\[ \T'(Y)=\{ M\in \ModR \mid M_\p=0 \textrm{ for all } \p\in \Spec{R}\setminus Y\}, \]
see~\cite[Example, p. 168]{St},
thus $\T(Y) = \T'(Y)$, which establishes the bijection between specialization closed subsets $Y$ and hereditary torsion pairs in $\rmod R$.

For statements (ii) and (iii), we refer to \cite[Proposition VI.3.6 and Exercise VI.24]{St} and \cite[Theorem 3.3.8]{EJ}.

Finally for (iv), we know from \cite[Lemma 4.5.2]{GT} that $(\T(Y)\cap\rfmod R ,\F(Y)\cap\rfmod R)$ is a torsion pair in $\rfmod R$ and that
\[ (\varinjlim (\T(Y)\cap\rfmod R),\varinjlim (\F(Y)\cap\rfmod R)) \]
is a torsion pair in $\rmod R$. Note that both $\T(Y)$ and $\F(Y)$ are closed under taking direct limits. In the case of $\F(Y)$ this follows from (iii). Hence
\[ \varinjlim (\T(Y)\cap\rfmod R) \subseteq \T(Y) \quad \textrm{and} \quad \varinjlim (\F(Y)\cap\rfmod R) \subseteq \F(Y), \]
and by Lemma~\ref{lem:comparing} we have equalities.
\end{proof}

\begin{remark} \label{rem:takahashi}
The bijections from Proposition~\ref{prop:Gabriel} can be reinterpreted in terms of the one-to-one-correspondence
\[ Y \mapsto \{M\in \rfmod R \mid \Ass{M}\subseteq Y\}, \]
established by Takahashi in \cite[Theorem 4.1]{Tak}, between all subsets of $\Spec R$ and the subcategories of $\rfmod R$ which are closed under submodules and extensions. Indeed, this correspondence restricts to a bijection $Y\mapsto \mathcal \{M\in \rfmod R \mid \Supp{M}\subseteq Y\}$ between the subsets of $\Spec R$ closed under specialization and the Serre subcategories (i.e. subcategories closed under submodules, factor modules and extensions) of $\rfmod R$, which in turn correspond bijectively to the hereditary torsion pairs in $\ModR$ via the assignment $\clS\mapsto \varinjlim \clS$, see \cite[Lemma 2.3]{Kcoh}. 

In fact,  specialization closed subsets of $\Spec R$ parametrize even all wide (or coherent) subcategories of $\rfmod R$, that is, all full abelian subcategories closed under extensions \cite[Theorem A]{Tak}, as well as all
narrow subcategories of  
$\rfmod R$, that is, all subcategories closed under extensions and cokernels \cite{SW}. 
\end{remark}

There is an alternative description of the class  $\{ M\in \rfmod R \mid \Ass{M}\subseteq Y \}$. Given a subset $Y \subseteq \Spec R$, we say that a module $M \in \rfmod R$ is \emph{$Y$-subfiltered} provided there exists a chain
\[ 0 = M_0 \subseteq M_1 \subseteq \dots \subseteq M_\ell = M \]
of submodules of $M$ such that for each $i = 0, \dots, \ell-1$, the module $M_{i+1}/M_i$ is isomorphic to a submodule of $R/\p_i$ for some $\p_i \in Y$. 

It was shown by Hochster (cf.\ \cite[Lemma 2.1]{PT})  that any module $M \in \rfmod R$ is $(\Ass M)$-subfiltered. Thus $\{M\in \rfmod R\mid \Ass{M}\subseteq Y\}$ is the subcategory of $\rfmod R$ given by all $Y$-subfiltered modules. Indeed,  If $0 \to N \to M \to M/N \to 0$ is a short exact sequence in $\rfmod R$, then $\Ass N \subseteq \Ass M$ and $\Ass M \subseteq \Ass N \cup \Ass M/N$, so the claim follows directly by Hochster's result.

\medskip
For our classification, we need to decide, which of the classes in $\rfmod{R}$ closed under submodules and extensions are torsion-free classes in $\rfmod R$. These again correspond bijectively to subsets of $\Spec{R}$ closed under specialization, as has recently been shown in \cite[Theorem 1]{SW}. We prefer to give a simple direct argument here: 

\begin{proposition} \label{prop:torsion_modR}
The assignment
\[ Y \mapsto (\T(Y)\cap\rfmod R, \F(Y)\cap\rfmod R), \]
using the notation from Proposition~\ref{prop:Gabriel}, gives a bijective correspondence between subsets $Y \subseteq \Spec{R}$ closed under specialization and torsion pairs in $\rfmod R$.
\end{proposition}

\begin{proof}
By Proposition~\ref{prop:Gabriel}, $(\T(Y)\cap\rfmod R, \F(Y)\cap\rfmod R)$ is clearly a torsion pair in $\rfmod R$ for every specialization closed set $Y$, and the assignment is injective since $\p \in Y$ \iff $R/\p \in \T(Y)$. We must prove the surjectivity.

To this end, suppose that $(\T,\F)$ is a torsion pair in $\rfmod R$. By \cite[Theorem~4.1]{Tak} (cf.~Remark~\ref{rem:takahashi}) there is a subset $X\subseteq \Spec R$ such that $\F = \{ M\in \rfmod R \mid \Ass{M}\subseteq X \}$. Denote by $Y$ the maximal specialization closed subset of $\Spec R$ disjoint from $X$. That is, $Y = \{\p\in\Spec R \mid V(\p) \cap X = \emptyset\}$. We claim that
\[ \T \subseteq \{ M\in\rfmod R \mid \Supp{M} \subseteq Y \}. \]
Indeed, given $\p \in X$, we have $R/\p \in \F$. Then for any $N \in \T$, $\Hom RN{R/\p} = 0$ implies $\Hom {R_\p}{N_\p}{k(\p)} = 0$, so the finitely generated $R_\p$-module $N_\p$ has no maximal submodules. That is, $N_\p = 0$ by the Nakayama Lemma (see e.g.~\cite[1.2.28]{EJ}). In particular, $\Supp N$ is specialization closed and disjoint from $X$, hence $\Supp N \subseteq Y$. This proves the claim. We have shown that
\[ \T \subseteq \T(Y)\cap\rfmod R \quad \textrm{and} \quad \F \subseteq \F(Y)\cap\rfmod R, \]
which by Lemma~\ref{lem:comparing} implies that $\T = \T(Y)\cap\rfmod R$ and $\F = \F(Y)\cap\rfmod R$.
\end{proof}

Let us now give a relation to $1$-cotilting modules, using results of Bazzoni, Buan and Krause.

\begin{proposition} \label{prop:1-cotilt}
Let $R$ be a (not necessarily commutative) right noetherian ring. Then the $1$-cotilting classes $\C$ in $\rmod R$ correspond bijectively to the torsion-free classes $\F$ in $\rfmod R$ containing $R$. The correspondence is given by the assignments
\[ \C \mapsto \F = \C \cap \rfmod R \quad \textrm{and} \quad \F \mapsto \varinjlim \F. \]
\end{proposition}

\begin{proof} 
This follows from~\cite[Theorem 1.5]{BK}, since all 1-cotilting modules are pure-injective by \cite{Ba}. See also~\cite[Theorem 8.2.5]{GT}.
\end{proof}

As a direct consequence, we get a characterization and a classification of $1$-cotilting classes in $\rmod R$ for $R$ commutative. Note that for $R$ non-commutative the torsion pair having as the torsion-free class a $1$-cotilting class need not be hereditary; see \cite[Theorem 2.5]{D}.

\begin{theorem} \label{thm:class-1-cot}
Let $R$ be a commutative noetherian ring and $\C \subseteq \rmod R$. Then $\C$ is $1$-cotilting \iff $\C$ is the torsion-free class in a faithful hereditary torsion pair $(\T,\C)$. In particular, the $1$-cotilting classes $\C$ in $\rmod R$ are parametrized by the subsets $Y$ of $\Spec R$ closed under specialization with $\Ass R \cap Y=\emptyset$. The parametrization is given by
\[
\C \mapsto \Spec R \setminus \Ass{(\C)} 
\quad \textrm{and} \quad
Y \mapsto \{ M \in \rmod R \mid \Ass M \cap Y = \emptyset \}.
\]
\end{theorem}

\begin{proof}
By Proposition~\ref{prop:1-cotilt}, $1$-cotilting classes in $\rmod R$ correspond bijectively to torsion-free classes in $\rfmod R$ containing $R$, which by Propositions~\ref{prop:Gabriel} and~\ref{prop:torsion_modR} and~\cite[Lemma 4.5.2]{GT} correspond bijectively to faithful hereditary torsion pairs in $\rmod R$. Composing the two assignments amounts to identifying a cotilting class $\C$ with the torsion-free part of the hereditary torsion pair. This shows the first part.

For the parametrization, we can use Proposition~\ref{prop:Gabriel}, as soon as we prove that
\[ \Ass{(\C \cap \rfmod R)} = \Ass{\C} \]
for any $1$-cotilting class $\C$. Clearly, $\Ass{(\C \cap \rfmod R)} \subseteq \Ass{\C}$. Conversely,
if $M \in \C$ and $\p\in\Ass M$, then $R/\p$ is embedded in $M$ and therefore $\{\p\}=\Ass{(R/\p)}$ is contained in $\Ass{(\C \cap \rfmod R)}$.
%
%
\end{proof}

Now, we will give a connection to tilting classes. For this purpose, we recall the concept of a transpose from~\cite{AUBR}.

\begin{definition} \label{def:Tr}
Let $C \in \rmod{R}$ and $P_1 \overset{f}\to P_0 \to C \to 0$ be a projective presentation in $\rmod{R}$. Then an \emph{Auslander-Bridger transpose} of $C$, denoted by $\Tr C$, is the cokernel of $f^*$, where $(-)^* = \Hom R-R$. That is, we have an exact sequence
\[ P_0^* \overset{f^*}\la P_1^* \la \Tr C \la 0. \]
\end{definition}

Note that by~\cite[Corollary 2.3]{AUBR}, $\Tr C$ is uniquely determined up to adding or splitting off a projective direct summand. The following lemma gives some homological formulae for  the transpose.

\begin{lemma} \label{lem:pdn}
Let $R$ be a (not necessarily commutative) left noetherian ring, and let $0 \ne U\in\lfmod R$ and $n \ge 0$ \st $\Ext {i}R{U}R = 0$ for all  $i = 0, 1, \dots, n$. 
 Then we have:
\begin{ii}
  \item $\pd R{\Tr{\Syz{n}{U}}} = n+1$;
  \item $\Ext nR{U}{-}$ and $\Tor 1R{\Tr{\Syz{n}{U}}}{-}$ are isomorphic functors.
  \item $\Ext 1R{\Tr{\Syz{n}{U}}}-$ and $\Tor nR-{U}$ are isomorphic functors.
\end{ii}
\end{lemma}

\begin{proof}
(i) Consider the beginning of a projective resolution of $U$:
\[ Q_{n+1} \overset{f_{n+1}}\la Q_n \overset{f_n}\la \dots \overset{f_1}\la Q_0 \la U \la 0.  \]
Denoting as in Definition~\ref{def:Tr} by $(-)^*$ the functor $\Hom R-R$, we get a sequence
\[
0 \longleftarrow \Tr{\Syz{n}{U}}
\longleftarrow Q^*_{n+1}
\overset{f^*_{n+1}}\longleftarrow Q^*_n
\overset{f^*_n}\longleftarrow \dots
\overset{f^*_1}\longleftarrow Q^*_0 \longleftarrow 0.
\]
which is exact by assumption. This shows that $\pd R{\Tr{\Syz{n}{U}}} \le n+1$. If $\pd R{\Tr{\Syz{n}{U}}} \le n$, then $f_1^*$ is a split monomorphism, so $f_1 = (f_1^*)^*$ is a split epimorphism, a contradiction.

(ii), (iii) These parts follow immediately using the well-known natural isomorphisms $\Hom RQN \cong Q^* \otimes_R N$ and $\Hom R{Q^*}M \cong M \otimes_R Q$ for all $M \in \rmod R$, $Q, N \in \lmod R$ with $Q$ finitely generated and projective.
\end{proof}

It follows that all $1$-cotilting classes over a one-sided noetherian ring are of cofinite type, that is, they are associated to $1$-tilting classes by the elementary duality: 

\begin{theorem} \label{thm:cofinite type}
Let $R$ be a (not necessarily commutative) left noetherian ring. The assignment $T\mapsto T^+$ induces a bijection between equivalence classes of $1$-tilting right $R$-modules and equivalence classes of $1$-cotilting left $R$-modules. 

In particular, given a $1$-cotilting class $\C$ in $\lmod R$, there is a class $\mathcal U\subseteq\lfmod R$ with $U^\ast=0$ for all $U\in\mathcal U$ \st
\[\mathcal C=\{ M\in \lmod R \mid \Hom{R}{U}{M}   =0 \textrm{ for all } U\in \mathcal U\}.\]
The preimage of $\C$ under the assignment above is then the $1$-tilting class
\begin{align*}
\mathcal D =& \{ M \in \rmod R \mid \Ext 1R{\Tr{U}}{M}=0 \textrm{ for all } U\in \mathcal U\} = \\
           & \{ M\in \ModR \mid {M}\otimes{U}=0 \textrm{ for all } U\in \mathcal U\}.
\end{align*}
\end{theorem}
 \begin{proof}
By a left-hand version of Proposition \ref{prop:1-cotilt} there is a torsion pair $(\mathcal U, \mathcal F)$ in $\lfmod R$ such that $R\in\mathcal F$ and $\mathcal C=\varinjlim \mathcal F= \{ M\in \lmod R \mid \Hom{R}{U}{M}   =0 \textrm{ for all } U\in \mathcal U\}$, see also \cite[Theorem 4.5.2]{GT}.
By Lemma \ref{lem:pdn}(i) and (ii) for $n=0$ the class $\mathcal S=\{\Tr U \mid U\in\mathcal U\}\subseteq\rfmod R$ consists of finitely presented modules of projective dimension one, and $\mathcal C=\mathcal S^\intercal$. 
Now apply Lemma \ref{lem:cofinite} and \ref{lem:pdn}(iii).
\end{proof}

Now we summarize our findings for the one-dimensional setting over commutative noetherian rings in the main theorem of the section.

\begin{theorem} \label{thm:class-dim1}
Let $R$ be a commutative noetherian ring. Then there are bijections between the following sets:
\begin{ii}
\item $1$-tilting classes $\D$ in $\rmod R$.
\item $1$-cotilting classes $\C$ in $\rmod R$.
\item Subsets $Y \subseteq \Spec R$ closed under specialization \st $\Ass R \cap Y = \emptyset$.
\item Faithful hereditary torsion pairs $(\T,\F)$ in $\rmod R$.
\item Torsion pairs $(\T',\F')$ in $\rfmod R$ with $R \in \F'$.
\end{ii}
\end{theorem}

\begin{proof}
Let us first explicitly state the bijections:

\begin{center}
\begin{tabular}{r@{$\ \to\ $}l@{\qquad}r@{$\ \mapsto\ $}l}
\multicolumn{2}{c}{Bijection} & \multicolumn{2}{c}{Assignment} \\ \hline
(i)   & (ii)  & $\D$   & $({^\perp \D} \cap \rfmod R)^\intercal$               \\
(ii)  & (iii) & $\C$   & $\Spec R \setminus \Ass{(\C \cap \rfmod R)}$                            \\
(iii) & (ii)  & $Y$    & $\{ M \in \rmod R \mid \Ass M \cap Y = \emptyset \}$  \\
(ii)  & (iv)  & $\C$   & $(\T,\F=\C)$                                                 \\
(iv)  & (v)   & $\F$  & $\F \cap \rfmod R$                                   \\
(v)   & (ii)  & $\F'$ & $\varinjlim \F'$                                     \\
\end{tabular}
\end{center}
The assignment in the first line of the table is bijective by Theorem~\ref{thm:cofinite type}. The second, third and fourth line in the table are covered by Theorem~\ref{thm:class-1-cot}. The fifth line follows from Propositions~\ref{prop:Gabriel} and~\ref{prop:torsion_modR}, while the sixth line is implied by Proposition~\ref{prop:1-cotilt}.
\end{proof}

We close this section with an equivalent, but more straightforward, parametrization of $1$-tilting classes in terms of the coassociated prime ideals and divisibility:

\begin{definition}
Let $R$ be a commutative noetherian ring.  

(1) Given an $R$-module $M$, a prime ideal $\p\in\Spec{R}$ is said to be {\it coassociated} to $M$ provided that $\p={\rm Ann}_R(M/U)$ for some submodule $U$ of $M$ \st the module $M/U$ is artinian over $R$. We denote by  $\Coass{M}$ the set of all prime ideals coassociated to $M$. For $\mathcal M\subseteq \rmod R$, we set $\Coass {\mathcal M} = \bigcup_{M \in \mathcal M} \Coass M$.

(2) Given a subset $Y\subseteq\Spec{R}$, an $R$-module $M$ is said to be $Y$-{\it divisible}  if $\p M=M$ for all $\p\in Y$. We denote by ${\mathcal D}(Y)$ the class of all $Y$-divisible $R$-modules.
\end{definition}

\begin{corollary}\label{cor:coass}
Let $R$ be a commutative noetherian ring. Then  the $1$-tilting classes $\mathcal D$ in $\rmod R$ are parametrized by the subsets $Y$ of $\Spec R$ closed under specialization with $\Ass R \cap Y=\emptyset$. The parametrization is given by
\[
\mathcal D \mapsto \Spec R \setminus \Coass{(\mathcal D)} 
\quad \textrm{and} \quad
Y \mapsto \mathcal \{ M \in \rmod R \mid \Coass M \cap Y = \emptyset \}.
\]
Moreover, \[\Coass \mathcal \{ M \in \rmod R \mid \Coass M \cap Y = \emptyset \} = \Ass \{ M \in \rmod R \mid \Ass M \cap Y = \emptyset \}.\]
\end{corollary}
\begin{proof}
Given  a subset $Y\subseteq\Spec R$ closed under specialization such that $\Ass R\cap Y=\emptyset$, we know from Theorem \ref{thm:class-dim1} and Theorem~\ref{thm:cofinite type} that the   corresponding  1-tilting  class  is 
$\mathcal D=\{ M \in \rmod R \mid {M}\otimes R/\p=0 \textrm{ for all } \p\in Y\}$.
Tensoring the exact sequence $0\to \p\to R\to R/\p\to 0$ by $M$ yields that $M\otimes_R R/\p$ is isomorphic to the cokernel of the embedding $\p M\to M$. So $\mathcal D=\mathcal{D}(Y)$.
Moreover, by \cite[2.2]{Zoe1} a module $M$ is $Y$-divisible if and only if $\Coass{M}\cap V(\p)=\emptyset$ for all $\p\in Y$. Since $Y$ is closed under specialization, this means that $\mathcal D(Y)=\{ M \in \rmod R \mid \Coass M \cap Y = \emptyset \}$. 

So, the assignment $Y\mapsto \mathcal D(Y)$ gives the desired bijection with inverse map
$\mathcal D \mapsto \Spec R \setminus \Coass{(\mathcal D)}$.
\end{proof}

\section{General cotilting classes}
\label{sec:general}

In this section, we classify all $n$-cotilting classes in $\rmod R$ where $R$ is an arbitrary commutative noetherian ring. In the next section, we will apply this classification to characterize all $n$-tilting classes in $\rmod R$.
 
Unfortunately, our methods do not seem to provide much information on the corresponding $n$-(co)tilting modules. Except for special classes of examples in~\cite[Chapters 5, 6 and 8]{GT} and \cite[\S5]{PT}, the only known way to construct, say, a cotilting module for a cotilting class $\C$, seems to be as in the proof of~\cite[Theorem 8.1.9]{GT}, using so-called special $\C$-precovers.

Let us first introduce the sequences of subsets of $\Spec R$ which will parametrize both $n$-tilting and $n$-cotilting classes for given $n \geq 1$.

\begin{definition} \label{def:cx}
In the following $(Y_1, \dots, Y_n)$ will always denote a sequence of subsets of $\Spec{R}$ \st
\begin{ii}
 \item $Y_i$ is closed under specialization for all $1 \leq i \leq n$;
 \item $Y_1 \supseteq Y_2 \supseteq \dots \supseteq Y_n$;
 \item $(\Ass{\Cosyz{i-1}R}) \cap Y_i = \emptyset$ for all $1 \leq i \leq n$;
\end{ii}
and $X_i$ will always denote $\Spec{R}\setminus Y_i$. 
 
For any such  $(Y_1, \dots, Y_n)$ we define the class of modules
\[ \C_{(Y_1, \dots, Y_n)} = \{ M \in \rmod{R} \mid (\Ass{\Cosyz{i-1}M}) \cap Y_{i} = \emptyset \textrm{ for all } 1 \leq i \leq n\} \]
\end{definition}

\begin{remark}\label{rem:cx}
Equivalently by Lemma \ref{lem:Bass}, we can write
\[ \C_{(Y_1, \dots, Y_n)} = \{ M \in \rmod R \mid \mu_{i-1}(\p,M) = 0 \textrm{ for all } 1 \leq i \leq n \textrm{ and } \p \in Y_i \}. \]
For $i \geq 1$, denote by $P_i$ the set of all prime ideals in $R$ of height $i - 1$. Since $P_1 \subseteq \Ass R$, the well-known properties of Bass invariants of finitely generated modules imply that $P_i \subseteq \Ass{\Cosyz{i-1}R} \subseteq X_i$ for all $1 \leq i \leq n$ (see e.g.~\cite[Proposition 9.2.13]{EJ}). 
In other words, (iii) implies (iii$^*$) where
\begin{ii}
\item[(iii$^*$)] $P_i \subseteq X_i$ for all $1 \leq i \leq n$.
\end{ii}
Since Gorenstein rings are characterized by the equality $P_i = \Ass{\Cosyz{i-1}R}$ for each $i \geq 1$ by~\cite[Theorem 18.8]{MATSUMURA}, it follows that (iii) is equivalent to (iii$^*$) when $R$ is Gorenstein. However, for general commutative noetherian rings condition (iii) may be more restrictive. In an extreme case, it may prevent existence of any non-trivial sequences $(Y_1, \dots, Y_n)$ as in the following example.
\end{remark}

\begin{example} \label{ex:non-cx}
Let $k$ be a field, $S = k[x,y]/(x^2, xy)$, and let $(R,\m,k)$ be the localization of $S$ at the maximal ideal $(x,y)$. It is easy to check that the ideal $(x) \subseteq R$ is simple, so $\m \in \Ass{R}$. Hence given any $(Y_1, \dots, Y_n)$ as in Definition~\ref{def:cx}, we necessarily have $Y_i = \emptyset$ for all $1 \leq i \leq n$ and $\C_{(Y_1, \dots, Y_n)} = \rmod R$. In view of the main theorem below, this implies that there are no non-trivial tilting or cotilting classes over this ring $R$.
\end{example}

Our next task is to prove that $\C_{(Y_1, \dots, Y_n)}$ are precisely the $n$-cotilting classes in $\rmod R$. The following definition and lemma will allow us to use induction on $n$.

\begin{definition} \label{def:stepdown}
For any cotilting module $C \in \rmod R$, the corresponding cotilting class $\C = {^\perp C}$ and $j \geq 1$, we define the class
\[ \C_{(j)} = {^\perp \Cosyz {j-1}C} = \{ M \in \rmod R \mid \Ext iRMC = 0 \textrm{ for all } i\geq j \}. \]
\end{definition}
Notice that $\C = \C_{(1)} \subseteq \C_{(2)} \subseteq \dots \subseteq \C_{(n)} \subseteq \C_{(n+1)} = \rmod R$ when $C$ is $n$-cotilting. 

\begin{lemma}\label{lem:stepdown}
Let $\C = {^\perp C}$ be an $n$-cotilting class. Then $\C_{(j)}$ is an $(n-j+1)$-cotilting class for any $j \leq n+1$.
\end{lemma}

\begin{proof}
The class $\C_{(j)}$ is closed under direct products by~\cite[Lemma 3.4]{BAZ} (see also~\cite[Proposition 8.1.5(a)]{GT}). The rest follows from the characterization of cotilting classes in~\cite[Corollary 8.1.10]{GT}. There, one uses the notion of cotorsion pairs introduced below in Definition~\ref{def:filtr_and_cotorsion}
\end{proof}

\begin{remark} \label{rem:C2}
If $D$ is another module with $\mathcal C = {}^{\perp}D$, then we can also use $D$ to compute $\mathcal C_{(j)}$ for each $j \geq 1$. Indeed, by dimension shifting, for each $M \in \rmod R$ we have $M \in \C_{(j)}$, if and only if $\Omega^{j-1}(M) \in \C$. So $ \mathcal C_{(j)}$ is uniquely determined by the class ${}^{\perp}C={}^{\perp}D.$ 

In particular, performing the construction  from \ref{def:stepdown} for the cotilting class $\mathcal C_{(2)}$, we obtain  $(\mathcal C_{(2)})_{(j)} = \mathcal C_{(j+1)} $ for all $j\ge 1$.
\end{remark}

Now we can state the main classification result of this section.

\begin{theorem} \label{thm:class-n-cot}
Let $R$ be a commutative noetherian ring and $n \ge 1$. Then the assignments
\begin{align*}
\Phi \colon && \C &\longmapsto (\Spec R \setminus \Ass{\C_{(1)}}, \dots,  \Spec R \setminus \Ass{\C_{(n)}}),   \\
\Psi \colon && (Y_1, \dots, Y_n) &\longmapsto \C_{(Y_1, \dots, Y_n)}
\end{align*}
give mutually inverse bijections between the sequences of subsets $(Y_1, \dots, Y_n)$ of $\Spec{R}$ satisfying the three conditions of Definition \ref{def:cx}, and the $n$-cotilting classes $\C$ in $\rmod{R}$.
\end{theorem}

We will prove the theorem in several steps. We start by proving that the map $\Psi$ is injective, but we postpone the proof of the fact that $\Psi$ is well-defined in the sense that each class of the form $\C_{(Y_1, \dots, Y_n)}$ is cotilting.

\begin{lemma} \label{lem:cx_unicity}
Let $(Y_1, \dots, Y_n)$ and $(Y'_1, \dots, Y'_n)$ be two sequences as in Definition~\ref{def:cx}. Then $\C_{(Y_1, \dots, Y_n)} = \C_{(Y'_1, \dots, Y'_n)}$ \iff $(Y_1, \dots, Y_n) = (Y'_1, \dots, Y'_n)$.
\end{lemma}

\begin{proof}
We only have to prove that $\C_{(Y_1, \dots, Y_n)} \ne \C_{(Y'_1, \dots, Y'_n)}$ whenever $(Y_1, \dots, Y_n) \ne (Y'_1, \dots, Y'_n)$. Thus suppose that there are $1 \le i \le n$ and $\p \in \Spec R$ \st $\p \in Y'_i \setminus Y_i$. By conditions~(ii) and (iii) of Definition~\ref{def:cx} for $Y'_i$, this implies
\[ \mu_{j}(\p,R) = \dim_{k(\p)} \Ext{j}{R_\p}{k(\p)}{R_\p} = 0 \qquad \textrm{ for all } 0 \le j \le i - 1 . \]
Denoting by $M$ an $(i-1)$-th syzygy module of $k(\p)$, we claim that $M \in \C_{(Y_1, \dots, Y_n)} \setminus \C_{(Y'_1, \dots, Y'_n)}$. Indeed, by Lemma \ref{lem:kp}(2) the only possible associated prime of a cosyzygy of $k(\p)$ is $\p$, so Corollary~\ref{cor:syzygies} and Remark \ref{rem:syzygy_indep} give us for each $0 \le j \le n - 1$:
\[
\Ass{\Cosyz{j}{M}} \subseteq
\begin{cases}
\bigcup_{k=0}^{i-2} \Ass{\Cosyz{j-k}{R}}               & \textrm{for } j < i-1    \\
\bigcup_{k=0}^{i-2} \Ass{\Cosyz{j-k}{R}} \cup \{\p\}   & \textrm{for } j \ge i-1. \\
\end{cases}
\]
Using Definition~\ref{def:cx}, one easily checks that $M \in \C_{(Y_1, \dots, Y_n)}$.

On the other hand, a straightforward dimension shifting argument based on the fact that $\Ext{j}{R_\p}{k(\p)}{R_\p} = 0$ for all $0 \le j \le i-1$ proved above yields
\[ \Ext{i-1}{R_\p}{k(\p)}{M_\p} \cong \Hom{R_\p}{k(\p)}{k(\p)} \ne 0, \]
so $\mu_{i-1}(\p, M) \not= 0$ by Lemma~\ref{lem:Bass} and $M \notin \C_{(Y'_1, \dots, Y'_n)}$.
\end{proof}

Next, we observe a consequence of the fact that every cotilting class is closed under taking direct limits (see~\cite[Theorem 8.1.7]{GT}).

\begin{lemma} \label{lem:tensoring}
Let $R$ be a commutative ring. Let $\C$ be a cotilting class in $\rmod{R}$, and let $M \in \C$ and $F$ be a flat $R$-module. Then $M \otimes_R F \in \C$. In particular, $M_\p \in \C$ for any $M \in \C$ and $\p \in \Spec{R}$.
\end{lemma}

\begin{proof}
By Lazard's theorem (see e.g.~\cite[Corollary 1.2.16]{GT}), we can express $F$ as a direct limit $F = \varinjlim_{i \in I} F_i$ of finitely generated free modules $F_i$. In particular, $M \otimes_R F_i \cong M^{n_i} \in \C$ for each $i \in I$. Since $\C$ is closed under taking direct limits by~\cite[Theorem 8.1.7]{GT}, we have $M \otimes_R F \cong \varinjlim_{i \in I} M \otimes_R F_i \in \C$. The last assertion follows since $M_\p \cong M \otimes_R R_\p$ and $R_\p$ is flat as an $R$-module. 
\end{proof}

The next observation gives us a relation between $\C$ and $\C_{(2)}$ (cf.\ Definition~\ref{def:stepdown} and Remark~\ref{rem:C2}).

\begin{lemma} \label{lem:cotilt_shift}
Let $\C$ be a cotilting class and
\[ 0 \la K \la L \la M \la 0 \]
be a short exact sequence \st $L \in \C$. Then $K \in \C$ \iff $M \in \C_{(2)}$.
\end{lemma}

\begin{proof}
Let $C$ be a cotilting module for $\C$. Then $\Ext iRKC \cong \Ext {i+1}RMC$ for each $i \ge 1$. The conclusion follows directly from the definition.
\end{proof}

Now we prove another part of Theorem~\ref{thm:class-n-cot}, namely that $\Psi \circ \Phi = id$. Again, we postpone for the moment the proof that the map $\Phi$ is well defined in the sense that the sequence $(\Spec R \setminus \Ass{\C_{(1)}}, \dots, \Spec R \setminus \Ass{\C_{(n)}})$ of subsets of $\Spec{R}$ satisfies for each $n$-cotilting class $\C$ the conditions in Definition~\ref{def:cx}.

\begin{proposition} \label{prop:k(p)_inside}
Let $n \ge 1$ and $\C$ be an $n$-cotilting class. Then  the following hold:
\begin{ii}
 \item If $\p \in \Ass{\C}$, then $k(\p) \in \C$.
 \item $\C$ is closed under taking injective envelopes.
 \item Define $X_i = \Ass{\C_{(i)}}$ and $Y_i = \Spec R \setminus X_i$ for $1 \le i \le n$. Then
 \[ \qquad \C = \{ M \in \rmod{R} \mid \Ass{E_{i-1}(M)} \cap Y_{i} = \emptyset \textrm{ for all } 1 \leq i \leq n\}. \]
\end{ii}
\end{proposition}

\begin{proof}
We will prove the statement by induction on $n$. More precisely, we will first show that (i) and (iii) hold for $n=1$, and that (i)~$\Rightarrow$~(ii)
for each $n \ge 1$.  Then we will prove the statements (i) and (iii) simultaneously by induction.

The proof of (i) for $n=1$: Suppose that $\p \in \Ass{\C}$. That is, $R/\p \subseteq M$ for some $M \in \C$. Lemma~\ref{lem:tensoring} then gives $k(\p) \subseteq M_{\p} \in \C$. By Theorem \ref{thm:class-1-cot}, $\C$ is a torsion-free class,  so $\C$ is closed under submodules and $k(\p) \in \C$.

(iii) for $n = 1$ is a straightforward consequence of Theorem \ref{thm:class-1-cot}.

(i)~$\Rightarrow$~(ii) for each $n \geq 1$:  By Lemma~\ref{lem:Bass} for $i=0$, for each $M \in \rmod R$, $E(M)$ is a direct sum of copies of the modules $E(R/\p)$ for $\p \in \Ass M$. So if $\p \in \Ass \C$, then $k(\p) \in \C$ by (i), and since $E(R/\p)$ is $k(\p)$-filtered by Lemma \ref{lem:kp}, also $E(R/\p) \in \C$. Thus $\C$ is closed under injective envelopes.

(i) for $n>1$: Suppose that $\p \in \Ass{\C}$. As above, we find $M \in \C$ \st $k(\p) \subseteq M$. To show that $k(\p) \in \C$, in view of Lemma~\ref{lem:cotilt_shift}, it suffices to prove that $M/k(\p) \in \C_{(2)}$. To this end, we know from Lemma \ref{lem:kp} that $\Ass{\Cosyz i{k(\p)}} \subseteq \{\p\}$ for each $i \ge 0$. Then Lemma~\ref{lem:ses}(iii) implies that
\[ \Ass{\Cosyz i{M/k(\p)}} \subseteq \Ass{\Cosyz iM} \cup \{\p\} \qquad \textrm{for each } i \ge 0. \]
However, $M \in \C \subseteq \C_{(2)}$, so condition (ii) for the ($n-1$)-cotilting class $\C_{(2)}$ and Lemma \ref{lem:cotilt_shift} give $\Ass{\Cosyz{1}M} \subseteq \Ass{\C_{(3)}}$, and similarly $\Ass{\Cosyz{i}M} \subseteq \Ass{\C_{(i+2)}}$ for all $0 \le i \le n-2$. Clearly $\p \in\Ass{\C}\subseteq \Ass{\C_{(i+2)}}$ for all $0 \le i \le n-2$ since $\C \subseteq \C_{(i+2)}$. Thus $\Ass{E_{i-2}({M/k(\p)})} \subseteq \Ass{\C_{(i)}} = \Ass{(\C_{(2)})_{(i-1)}}$ for all $2 \leq i \leq n$. Condition (iii) for the ($n-1$)-cotilting class $\C_{(2)}$ then gives $M/k(\p) \in \C_{(2)}$, so $k(\p) \in \C$ by Lemma \ref{lem:cotilt_shift}.

(iii) for $n > 1$: Using conditions (i) and (ii) for $n$ and Lemma \ref{lem:kp}, we obtain the implications 
\[
E(R/\p) \in \C \quad \Rightarrow \quad
\p \in \Ass\C \quad \Rightarrow \quad
k(\p) \in \C \quad \Rightarrow \quad
E(R/\p) \in \C. \]

Also, condition (ii) for $n$ and Lemma~\ref{lem:cotilt_shift} imply that a module $M$ belongs to $\C$, \iff $E(M) \in \C$ and $\Cosyz{}M \in \C_{(2)}$.
Since for each module $M$, the indecomposable direct summands of $E(M)$ are precisely the $E(R/\p)$ for $\p \in \Ass M$, we infer that $E(M) \in \C$ \iff $\Ass M \subseteq \Ass{\C} = X_1$.

We now apply condition (iii) to the ($n-1$)-cotilting class $\C_{(2)}$. By Remark \ref{rem:C2} we obtain
\[ \C_{(2)} = \{ L \in \rmod{R} \mid \Ass{E_{i-2}(L)} \subseteq X_{i} \; \mbox{for all} \; 2 \leq i \leq n\}. \]
In particular, $\Cosyz{}M \in \C_{(2)}$ \iff $\Ass{E_{i-1}(M)} \subseteq X_{i}$ for all $2 \leq i \leq n$, and the conclusion follows.
\end{proof}

Let us summarize what has been done so far. We have proved that the assignment $\Psi$ in Theorem~\ref{thm:class-n-cot} is injective, and that $\Psi \circ \Phi = id$. We are left to show that each sequence of subsets in the image of $\Phi$ meets the requirements of Definition~\ref{def:cx}, and that each class obtained by an application of $\Psi$ is actually cotilting. We start with the former statement, which is easier.  

\begin{lemma}\label{lem:generaliz}
Let $n \ge 1$ and $\C$ be an $n$-cotilting class. If we put $X_i = \Ass{\C_{(i)}}$ and $Y_i = \Spec R \setminus X_i$ for $1 \le i \le n$, then the sequence $(Y_1, \dots, Y_n)$ of subsets of $\Spec{R}$ satisfies conditions \emph{(i)-(iii)} in Definition~\ref{def:cx}.
\end{lemma}

\begin{proof}
Condition (ii) is clear from the inclusions $\C = \C_{(1)} \subseteq \dots \subseteq \C_{(n)}$. Condition (iii) holds for $i = 1$ because $R \in \C$; for $1 < i \leq n$ it follows by induction using Lemma \ref{lem:cotilt_shift} and Proposition \ref{prop:k(p)_inside}(ii). 

In order to show (i), we prove that each $X_i$ is closed under generalization. Let $\p \in X_i$. Then $k(\p) \in \C_{(i)}$ by Proposition~\ref{prop:k(p)_inside}(i). Hence $E(k(\p)) \in \C_{(i)}$ and  $E_R(k(\p)) \cong E_{R_\p}(k(\p)) = E_R(R/\p)$, by Lemma \ref{lem:kp}. This implies that $\C_{(i)}$ contains an injective cogenerator for $\rmod{R_\p}$. Given any $\q \subseteq \p$ in $\Spec{R}$, $E(R/\q)$ is an injective $R_\p$-module (see e.g.\  \cite[Theorem 3.3.8(1)]{EJ}), so $E(R/\q)$ is a direct summand in $E_R(R/\p)^I$ for some set $I$. But $\C_{(i)}$ is closed under arbitrary direct products and direct summands, hence also $E(R/\q) \in \C_{(i)}$ and $\q \in X_i = \Ass{\C_{(i)}}$.
\end{proof}

Finally, we are going to prove that each class $\C = \C_{(Y_1, \dots, Y_n)}$ as in Definition~\ref{def:cx} is $n$-cotilting. We require a few definitions first.

\begin{definition} \label{def:filtr_and_cotorsion}
A class $\C$ of modules is called \emph{definable} if it is closed under direct products, direct limits and pure submodules. A pair $(\C, \mathcal D)$ of classes of modules is a \emph{cotorsion pair} if
\begin{align*}
\mathcal D &= \{ D \in \rmod{R} \mid \Ext 1RCD = 0 \textrm{ for all } C \in \C \} \quad \textrm{and} \\
\C &= \{ C \in \rmod{R} \mid \Ext 1RCD = 0 \textrm{ for all } D \in \mathcal D \}.
\end{align*}
A cotorsion pair $(\C, \mathcal D)$ is \emph{hereditary} if $\C$ is closed under taking syzygies.
\end{definition}

The following characterization of $n$-cotilting classes will be useful for completing our task:

\begin{proposition} \label{prop:char_cotilting}
Let $n \ge 0$ and $\C$ be a class of modules. Then $\C$ is $n$-cotilting, \iff all of the following conditions are satisfied:
\begin{ii}
 \item $\C$ is definable,
 \item $R \in \C$ and $\C$ is closed under taking extensions and syzygies (in conjunction with (i), this only says that $\C$ is resolving in $\rmod R$),
 \item each $n$-th syzygy module belongs to $\C$.
\end{ii}
\end{proposition}

\begin{proof}
If $\C$ is $n$-cotilting, then $\C$ is definable by~\cite[Theorem 8.1.7]{GT}. Clearly $R \in \C$, and there is a hereditary cotorsion pair of the form $(\C,\C ^\perp)$ such that the class $\C ^\perp$ consists of modules of injective dimension $\leq n$ by \cite[Theorem 8.1.10]{GT}. This implies conditions (ii) and (iii).

Assume on the other hand that (i)-(iii) hold. Using~\cite[Lemma 1.2.17]{GT}, we can construct for each $M \in \C$ a well-ordered chain
\[ 0 = M_0 \subseteq M_1 \subseteq M_2 \subseteq \cdots \subseteq M_\alpha \subseteq M_{\alpha+1} \subseteq \cdots M_\sigma = M, \]
in $\C$ consisting of pure submodules of $M$ \st $\card{M_{\alpha+1}/M_\alpha} \le \card{R} + \aleph_0$ for each $\alpha<\sigma$ and $M_\beta = \bigcup_{\alpha<\beta} M_\alpha$ for every limit ordinal $\beta\le\sigma$.
Note that definable classes are closed under taking pure epimorphic images by~\cite[Theorem 3.4.8]{PREST}. Thus also each subfactor $M_{\alpha+1}/M_\alpha$ belongs to $\C$. In particular, it follows easily that $M \in \C$ \iff $M$ is $\mathcal S$-filtered, where $\mathcal S$ is a representative set for the modules in $\C$ of cardinality $\leq \card{R} + \aleph_0$. Since clearly $R \in \mathcal S$, we can use~\cite[Corollary 3.2.4 and Lemma 4.2.10]{GT} to infer that $\C$ fits into a hereditary cotorsion pair $(\C, \mathcal D)$. A simple dimension shifting using condition (iii) tells us that all modules in $\mathcal D$ have injective dimension at most $n$. Thus, $\C$ is an $n$-cotilting class by~\cite[Corollary 8.1.10]{GT}.
\end{proof}

Now we are ready to give the last piece of the proof of Theorem~\ref{thm:class-n-cot}.

\begin{proposition} \label{prop:cx_cotilt}
Let $(Y_1, \dots, Y_n)$ be a sequence of subsets of $\Spec{R}$ meeting the requirements of Definition~\ref{def:cx}. Then the class $\C = \C_{(Y_1, \dots, Y_n)}$ is $n$-cotilting.
\end{proposition}

\begin{proof}
We use the characterization of $n$-cotilting classes from Proposition~\ref{prop:char_cotilting}. Clearly, $R \in \C$ by the assumptions on $(Y_1, \dots, Y_n)$. Conditions (ii) and (iii) of Proposition~\ref{prop:char_cotilting} then follow easily from Lemma~\ref{lem:ses} and Corollary~\ref{cor:syzygies} (see also Remark~\ref{rem:syzygy_indep}). Thus, it only remains to prove that $\C$ is definable.

To this end, note first that for a family of modules, the product of injective coresolutions of the modules is a (possibly non-minimal) injective coresolution of the product of the modules. Using the fact that $Y_i$ is closed under specialization for every $i$, Proposition~\ref{prop:Gabriel} tells us that the class
\[ \mathcal{E}_i = \{ E \in \ModR \mid E \textrm{ is injective and } \Ass{E} \cap Y_i = \emptyset \} \]
is closed under products for every $i$ since it is precisely the classes of all injective $R$-modules contained in the torsion-free class $\F(Y_i)$. Hence $\C$ is closed under products itself, using Definition~\ref{def:cx} and Lemma~\ref{lem:Bass}.

Assume next that $M \in \C$ and $K \subseteq M$ is a pure submodule. To prove that $K \in \C$, we must show that for each $1 \le i \le n$ and $\p \in Y_i$, we have
\[ \mu_i(\p,K) = \dim_{k(\p)} \Ext{i}{R_\p}{k(\p)}{K_\p} = 0. \]
Since the embedding $K \subseteq M$ is a direct limit of split monomorphisms and localizing at $\p$ preserves direct limits, also the embedding $K_\p \subseteq M_\p$ is pure. The conclusion that $\Ext{i}{R_\p}{k(\p)}{K_\p} = 0$ then follows from the fact that $k(\p)$ is a finitely generated $R_\p$-module and thus the class
\[ \{ N \in \rmod{R_\p} \mid \Ext{i}{R_\p}{k(\p)}{N} = 0 \} \]
is definable in $\rmod{R_\p}$, see~\cite[Example 3.1.11]{GT}.

The proof that $\C$ is closed under direct limits is similar. Namely for each $1 \le i \le n$ and $\p \in Y_i$, the class
\[ \{ M \in \rmod{R} \mid \Ext{i}{R_\p}{k(\p)}{M_\p} = 0 \} \]
is the kernel of the composition of two direct limit preserving functors: the localization at $\p$ and the functor $\Ext{i}{R_\p}{k(\p)}{-}$; and $\C$ is the intersection of all these classes.
\end{proof}

\begin{proof}[Proof of Theorem~\ref{thm:class-n-cot}]
Lemma~\ref{lem:generaliz} and Proposition~\ref{prop:cx_cotilt} show that $\Phi$ assigns to each $n$-cotilting class a sequence satisfying the conditions of Definition~\ref{def:cx}, and conversely that $\Psi$ assigns to each such sequence an $n$-cotilting class. Further, we have proved in Lemma~\ref{lem:cx_unicity} and Proposition~\ref{prop:k(p)_inside} that $\Psi$ is injective and $\Psi \circ \Phi = id$. Thus, $\Phi$ and $\Psi$ are mutually inverse bijections.
\end{proof}

We conclude our discussion by two consequences. We clarify the effect of passing from $\C$ to $\C_{(j)}$ in the sense of Definition~\ref{def:stepdown} on the corresponding filtrations of subsets of the spectrum:

\begin{corollary} \label{cor:cx_shifting}
Let $(Y_1, \dots, Y_n)$ be as in Definition~\ref{def:cx}. Then for any natural number $1 \leq j \leq n$ we have $\mathcal (\C_{(Y_1, \dots, Y_n)})_{(j)} = \C_{(Y_{j}, \dots, Y_n)}$.
\end{corollary}

\begin{proof}
Since we now know that $\C_{(Y_1, \dots, Y_n)}$ is an $n$-cotilting class, the statement follows directly from Remarks~\ref{rem:cx} and~\ref{rem:C2}.
\end{proof}

Further, we show that the dimension shifting in the sense of Definition~\ref{def:stepdown} works nicely also at the level of cotilting modules.

\begin{corollary} \label{cor:module_shifting}
Let $C$ be an $n$-cotilting module ($n \geq 2$) with the corresponding cotilting class given by $(Y_1, \dots, Y_n)$. Then $D = \Cosyz{}C \oplus \bigoplus_{\p\in X_2} E(R/\p)$ is an $(n-1)$-cotilting module with the corresponding cotilting class given by $(Y_2, \dots, Y_n)$.
\end{corollary}

\begin{proof}
Denote $\C = {^\perp C}$ the cotilting class. Clearly ${^\perp D} = {^\perp \Cosyz{}C} = \C_{(2)}$ which is the $(n-1)$-cotilting class given by $(Y_2, \dots, Y_n)$ by Corollary \ref{cor:cx_shifting}. 
 
Obviously, $D$ has injective dimension $\leq n-1$, so (C1) holds. Condition (C2) also holds for $D$ since for any $i \geq 1$ and any cardinal $\kappa$ we have $D \in \mathcal C_{(2)}$ by Lemma \ref{lem:cotilt_shift} and Proposition \ref{prop:k(p)_inside}, hence $D^\kappa \in \mathcal C_{(2)} = {^\perp D}$, and $\Ext{i}{R}{D^\kappa}{D} = 0$. To prove (C3), it is by~\cite[Lemma 3.12]{BAZ} enough to show that $\mathcal C_{(2)} \subseteq \Cog D$, that is, each $M \in \mathcal C_{(2)}$ is cogenerated by $D$. We will show more, namely that
\[ \{ M \in \rmod{R}\mid \Ass{M} \subseteq X_2 \} \subseteq \Cog D. \]
Indeed, taking any $M$ with $\Ass{M} \subseteq X_2$, we have
\[
M \subseteq E(M) = \bigoplus_{\p \in \Ass{M}}E(R/\p)^{(\mu_0(\p,M))} \subseteq \prod_{\p \in X_2}E(R/\p)^{\mu_0(\p,M)} \in \Cog D.
\qedhere
\]
\end{proof}

\section{The main theorem}
\label{sec:tgeneral}

We are now going to prove that the correspondence $T \mapsto T^+$ induces a bijection between the equivalence classes of $n$-tilting and $n$-cotilting modules. 
This correspondence together with Theorem \ref{thm:class-n-cot} will then rather quickly yield a proof of our main classification result.
  
We first need a translation of the definition of $\C_{(Y_1, \dots, Y_n)}$ in a homological condition.

\begin{lemma} \label{lem:cotilt_hom_alg}
Let $Y \subseteq \Spec R$ be specialization closed, $M \in \rmod R$ and $i \ge 0$ be an integer. Then the following are equivalent:
\begin{ii}
 \item $\mu_i(\p, M) = 0$ for each $\p \in Y$;
 \item $\Ext iR{R/\p}M = 0$ for each $\p \in Y$.
\end{ii}
\end{lemma}

\begin{proof}
If $\Ext iR{R/\p}M = 0$ for each $\p \in Y$, then the isomorphism
\[ 0 = (\Ext iR{R/\p}M)_\p \cong \Ext i{R_\p}{k(\p)}{M_\p}, \]
together with Lemma~\ref{lem:Bass} yield $\mu_i(\p,M) = 0$ for all $\p \in Y$.

Conversely, suppose that $\mu_i(\p, M) = 0$ for each $\p \in Y$ and consider the beginning of an injective coresolution of~$M$:
\[ 0 \la M \la E_0(M) \la E_1(M) \la \cdots \la E_{i-1}(M) \la E_i(M). \]
Then each element of $\Ext iR{R/\p}M$ is represented by a coset of some homomorphism $f \in \Hom R{R/\p}{E_i(M)}$. If $\p \in Y$, then on one hand $\Img f$ is an $R/\p$-module, so $\Ass{(\Img f)} \subseteq V(\p) \subseteq Y$. On the other hand, $\Ass{(\Img f)} \subseteq \Ass{E_i(M)} \subseteq \Spec R \setminus Y$ by Lemma~\ref{lem:Bass}. Thus, $f = 0$, and $\Ext iR{R/\p}M = 0$ as well.
\end{proof}

Now we are in a position to state and prove our main classification result.

\begin{theorem} \label{thm:class-main}
Let $R$ be a commutative noetherian ring and $n \ge 1$. Then there are bijections between:
\begin{ii}
  \item  Sequences $(Y_1, \dots, Y_n)$ of subsets of $\Spec{R}$ as in Definition~\ref{def:cx};
  \item $n$-tilting classes $\T \subseteq \rmod R$;
  \item $n$-cotilting classes $\C \subseteq \rmod R$.
\end{ii}
The bijections assign to $(Y_1, \dots, Y_n)$ the $n$-tilting class
\begin{align*}
\T =& \{ M \in \rmod R \mid \Tor {i-1}R{R/\p}M = 0 \textrm{ for all } i = 1, \dots, n \textrm{ and } \p \in Y_i \} = \\
    & \{ M \in \rmod R \mid \Ext 1R{\Tr{\Syz{i-1}{R/\p}}}M = 0 \textrm{ for all } i = 1, \dots, n \textrm{ and } \p \in Y_i \},
\end{align*}
and the $n$-cotilting class
\begin{align*}
\C =& \{ M \in \rmod R \mid \Ext {i-1}R{R/\p}M = 0 \textrm{ for all } i = 1, \dots, n \textrm{ and } \p \in Y_i \} = \\
    & \{ M \in \rmod R \mid \Tor 1R{\Tr{\Syz{i-1}{R/\p}}}M = 0 \textrm{ for all } i = 1, \dots, n \textrm{ and } \p \in Y_i \}.
\end{align*}
\end{theorem}

\begin{proof}
Let $(Y_1, \dots, Y_n)$ be as in Definition~\ref{def:cx} and $\C = \C_{(Y_1, \dots, Y_n)}$. Then
\[ \C = \{ M \in \rmod R \mid \Ext {i-1}R{R/\p}M = 0 \textrm{ for all } i = 1, \dots, n \textrm{ and } \p \in Y_i \} \]
by Lemma~\ref{lem:cotilt_hom_alg}. In particular we have
\[ \Ext {i-1}R{R/\p}R = 0 \qquad \textrm{ for all } i =  1, \dots, n \textrm{ and } \p \in Y_i, \]
since $\C$ is a cotilting class by Proposition~\ref{prop:cx_cotilt}. Thus, the expression of $\C$ in terms of the Tor-groups follows from Lemma~\ref{lem:pdn}(ii) (applied for $U = R/\p$, where $i = 1, \dots, n$ and $\p \in Y_i$), and the fact that we have a bijection between (i) and (iii) is an immediate consequence of Theorem~\ref{thm:class-n-cot}. The bijection between (ii) and (iii) is a consequence of Lemmas~\ref{lem:pdn}(iii) and~\ref{lem:cofinite}.
\end{proof}

In fact, the Ext and Tor orthogonals above for $\T$ and $\C$, respectively, can be taken \wrt (typically considerably smaller) sets of finitely generated modules. For a given sequence $(Y_1, \dots, Y_n)$, let us denote for each $i$ by $\bar Y_i$ the set of minimal elements in $Y_i$ \wrt inclusion. Since $(\Spec R,\subseteq)$ satisfies the descending chain condition, for each $\p \in Y_i$ there exists $\q \in \bar Y_i$ \st $\q \subseteq \p$. We claim that

\begin{corollary} \label{cor:optimized}
With the notation of Theorem~\ref{thm:class-main}, the class $\C$ equals
\[
\{ M \in \rmod R \mid \Tor 1R{\Tr{\Syz{i-1}{R/\p}}}M = 0 \textrm{ for all } i = 1, \dots, n \textrm{ and } \p \in \bar{Y}_i \}
\]
and the class $\T$ equals
\[
\{ M \in \rmod R \mid \Ext 1R{\Tr{\Syz{i-1}{R/\p}}}M = 0 \textrm{ for all } i = 1, \dots, n \textrm{ and } \p \in \bar{Y}_i \}.
\]
\end{corollary}

\begin{proof}
Let us provisionally denote the above candidate for $\C = \C_{(Y_1, \dots, Y_n)}$ by $\C'$. We shall prove that $\C' = \C$ by induction on the length $n$ of the sequence $(Y_1, \dots, Y_n)$.

First of all we claim that $\C'$ is $n$-cotilting. If $n = 1$, then $\C' = \bigcap_{\p \in \bar{Y}_1}\Tr{R/\p}^\intercal$ since $\pd{R}{\Tr{R/\p}} \le 1$ by Lemma~\ref{lem:pdn}(i). Hence $\C'$ is a $1$-cotilting class by Lemma~\ref{lem:cofinite}(i). If $n > 1$, then by induction hypothesis the class
\[
\{ M \in \rmod R \mid \Tor 1R{\Tr{\Syz{i-1}{R/\p}}}M = 0 \textrm{ for all } i = 1, \dots, n-1 \textrm{ and } \p \in \bar{Y}_i \}\]
equals $\C_{(Y_1, \dots, Y_{n-1})}$. Using the fact that $\bar{Y}_n \subseteq Y_i$ for all $i = 1, \dots, n-1$ and the description of $\C_{(Y_1, \dots, Y_{n-1})}$ from Theorem~\ref{thm:class-main}, it follows that
\[
\C' = \C_{(Y_1, \dots, Y_{n-1})} \cap \bigcap_{\p \in \bar{Y}_n} \Tr{\Syz{n-1}{R/\p}}^\intercal.
\]
Hence $\C'$ is $n$-cotilting by Lemmas~\ref{lem:pdn}(i) and~\ref{lem:cofinite}(i) again.

Now clearly $\C' \supseteq \C_{(Y_1, \dots, Y_n)}$. Thus, Theorem~\ref{thm:class-n-cot} implies that $\C' = \C_{(Y'_1, \dots, Y'_n)}$ for some sequence $(Y'_1, \dots, Y'_n)$ of specialization closed sets \st $Y'_i \subseteq Y_i$ for each $i$. On the other hand, 
since $\bar Y_i\subseteq Y_i$ and $R\in\mathcal C$, we have $\Ext{i-1}{R}{R/\p}{R}=0$ for all $1\le i\le n$ and $\p\in\bar Y_i$. Combining Lemma~\ref{lem:pdn}(ii) with the proof of (ii) $\Rightarrow$ (i) in Lemma~\ref{lem:cotilt_hom_alg}, we infer that  
$\mu_{i-1}(\p,M)=0$ for each $M \in \C'$ and $\p \in \bar Y_i$. In particular $Y'_i = \Spec R \setminus \Ass\C'_{(i)} \supseteq \bar Y_i$ for each $i = 1, \dots, n$ by Remark~\ref{rem:C2} and Corollary~\ref{cor:syzygies}. Since the $Y'_i$ are specialization closed, it follows that $Y'_i = Y_i$. The claim for $\T$ is a consequence of Lemma~\ref{lem:cofinite}(ii).
\end{proof}

In view of Lemma~\ref{lem:resolv}, Theorem \ref{thm:class-main} also yields a classification of the resolving classes in $\rfmod R$ consisting of modules of bounded projective dimension:

\begin{corollary} \label{cor:parres}
Let $R$ be a commutative noetherian ring and $n \ge 1$. Then there is a bijection between:
\begin{ii}
  \item Sequences $(Y_1, \dots, Y_n)$ of subsets of $\Spec{R}$ as in Definition~\ref{def:cx}; 
  \item resolving subclasses $\mathcal S$ of $\rfmod R$ consisting of modules of projective dimension $\leq n$.
\end{ii}
The bijection assigns to a sequence $(Y_1, \dots, Y_n)$ the class of all direct summands of finitely $\mathcal E$-filtered modules where (with the notation of Corollary~\ref{cor:optimized})
\[
\mathcal E = \{ \Tr{\Syz{i-1}{R/\p}} \mid i = 1, \dots , n \textrm{ and } \p \in \bar Y_i \} \cup \{ R \}.
\]
\end{corollary}
\begin{proof} Let $\mathcal T$ be the $n$-tilting class corresponding to the sequence $(Y_1, \dots, Y_n)$ by Theorem \ref{thm:class-main}. By Lemma \ref{lem:resolv}, $\T$ also corresponds to the resolving class $\mathcal S = {}^\perp \mathcal T \cap \rfmod R$. Using Corollary~\ref{cor:optimized}, we have $\T = \{ M \in \rmod R \mid \Ext1R{\mathcal E}M = 0\}$. Hence ${}^\perp \mathcal T$ is the class of all direct summands of $\mathcal E$-filtered modules by \cite[3.2.4]{GT}. Then $\mathcal S$ the class of all direct summands of finitely $\mathcal E$-filtered modules by Hill's Lemma \cite[4.2.6]{GT}.
\end{proof}

Another consequence of Theorem \ref{thm:class-main} reveals a remarkable lack of module approximations by resolving classes in $\rfmod R$ in the local case.

Given two classes $\mathcal A \subseteq \mathcal C \subseteq \rmod R$, we say that $\mathcal A$ is \emph{special precovering} in $\mathcal C$ provided that for each module $M \in \mathcal C$ there exists an exact sequence $0 \to B \to A \overset{f}\to C \to 0$ in $\mathcal C$ such that $A \in \mathcal A$ and $\Ext1R{A'}B = 0$ for each $A' \in \A$. The map $f$ is called a \emph{special $\A$-precover} of $C$.

Special precovering classes in $\rmod R$ are abundant: for example, if $\mathcal T$ is any tilting class, then the class ${}^\perp \mathcal T$ is special precovering in $\rmod R$, see \cite[5.1.16]{GT}. One might expect that $\mathcal S  = {}^\perp \mathcal T \cap \rfmod R$ will then be special precovering in $\rfmod R$. However, if $R$ is local then this occurs only in the trivial cases when $\mathcal T = \rmod R$ or $\mathcal S = \rfmod R$: 

\begin{corollary} \label{cor:rarespec}
Let $R$ be a commutative noetherian local ring and $\mathcal S$ be a resolving class consisting of modules of bounded projective dimension. Then the following are equivalent:
\begin{ii}
  \item $\mathcal S$ is special precovering in $\rfmod R$;
  \item either $\mathcal S$ is the class of all free modules of finite rank (and the $\mathcal S$-precovers can be taken as the projective covers), or else $R$ is regular and $\mathcal S = \rfmod R$.     
\end{ii} 
\end{corollary}
\begin{proof} We only have to prove that (i) implies (ii): Let $\mathcal T = \mathcal S ^\perp$. Then $\mathcal T$ is a tilting class by Lemma \ref{lem:resolv}. 
If $\mathcal T = \rmod R$, then $\mathcal S$ is the class of all free modules of finite rank and the claim is clear.

Otherwise, consider the sequence $(Y_1,\dots,Y_n)$ corresponding to $\mathcal T$ by Theorem \ref{thm:class-main}. Let $\p \in Y_1$. Then for each $M \in \mathcal T$, we have $R/\p \otimes _R M = 0$ and $\p M = M$ by Theorem \ref{thm:class-main}. The Nakayama Lemma thus gives $\mathcal T \cap \rfmod R = 0$. 

Let $C \in \rfmod R$. By (i), we have an exact sequence $0 \to B \to A \to C \to 0$ with $A \in \mathcal S$ and $B \in \mathcal T \cap \rfmod R$, hence $B = 0$ and $C \in \mathcal S$. Thus $\mathcal S = \rfmod R$, and $R$ has finite global dimension.          
\end{proof}

\begin{remark} \label{rem:tak}
In the particular case of henselian Gorenstein local rings, there is a more complete picture available. By \cite{Taktak}, the only resolving (special) precovering classes in $\rfmod R$ are (1) the class of all free modules of finite rank, (2) the class of all maximal Cohen-Macaulay modules, and (3) $\rfmod R$.
\end{remark}

\section{Cotilting over Gorenstein rings and Cohen-Macaulay modules}
\label{sec:Hochster}

In this final section, we will restrict ourselves to the particular setting of Gorenstein rings, and later even regular rings. We generalize some results from \cite{TP}, but our main concern is the relation to the existence of finitely generated Cohen-Macaulay modules and, in particular, to Hochster's Conjecture~E from~\cite{H}. The main outcome here is Theorem~\ref{thm:Lp_to_Kp}, which gives new information on properties of the (conjectural) maximal Cohen-Macaulay modules.

\subsection{Cotilting classes over Gorenstein rings}
\label{subsec:cotilting_Gorenstein}

We start by considering torsion products of injective modules over Gorenstein rings. Recall that $R$ is \emph{Gorenstein}, if $R$ is commutative noetherian and $\id {R_{\p}}{R_{\p}} < \infty$ for each $\p \in \Spec R$.

\begin{lemma} \label{lem:torinj} Let $R$ be a Gorenstein ring, $\p \in \Spec R$, $k = \htt \p$, and $M \in \rmod R$.
\begin{ii}
\item Let $\q \in \Spec R$ and $i \geq 0$. Then $\Tor iR{E(R/\p)}{E(R/\q)} \neq 0$, if and only if $\p = \q$ and $i = k$.    
\item $\fd R{E(R/\p)} = k$.
\item $\p \notin \Ass M$ \iff $\mu_{0}(\p,M) = 0$ \iff $\Tor kR{E(R/\p)}{M} = 0$.
\item Let $i$ be an integer \st $0 < i \leq k$ and suppose that $\mu_{i-1}(\p,M) = 0$. Then $\Tor{k-i}R{E(R/\p)}M \cong \Tor kR{E(R/\p)}{\Cosyz iM}$. In particular, we have $\mu_i(\p,M) = 0$ \iff $\Tor {k-i}R{E(R/\p)}{M} = 0$.
\end{ii}
\end{lemma}

\begin{proof}
(i) Let $i \geq 0$. If $r \in \p \setminus \q$, then the multiplication by $r$ is locally nilpotent on $E(R/\p)$, but an isomorphism on $E(R/\q)$. So both is true of the endomorphism of $\Tor iR{E(R/\p)}{E(R/\q)}$ given by the multiplication by $r$. This is only possible when $\Tor iR{E(R/\p)}{E(R/\q)} = 0$. (Note that this argument does not need the Gorenstein assumption). 

For the remaining case of $\p = \q$, we can assume that $R$ is local by \cite[Theorem 3.3.3]{EJ}; then the result is a consequence of \cite[Theorem 9.4.6]{EJ}.

(ii) This is proved in \cite[Proposition 5.1.2]{X}. 

(iii) The first equivalence is just a reminder of Lemma~\ref{lem:Bass}. For the second, assume $\mu_{0}(\p,M) = 0$. By Lemma \ref{lem:Bass}, the indecomposable decomposition of $E(M)$ does not contain any copy of $E(R/\p)$, so $\Tor kR{E(R/\p)}{E(M)} = 0$ by part (i). Since $\fd R{E(R/\p)} = k$ by part (ii), the kernel of the functor $\Tor kR{E(R/\p)}{-}$ is closed under submodules, so $\Tor kR{E(R/\p)}{M} = 0$. 

Conversely, if $\Tor kR{E(R/\p)}{M} = 0$ and $\p \in \Ass M$, then $\Tor kR{E(R/\p)}{R/\p} = 0$ by part (ii), so localizing at $\p$ we have $\Tor k{R_{\p}}{E(k(\p))}{k(\p)} = 0$, see \cite[2.1.11]{EJ}. So $\Tor k{R_{\p}}{E(k(\p))}{E(k(\p))} = 0$, because $E(k(\p))$ is a $\{ k(\p) \}$-filtered $R_{\p}$-module by Lemma \ref{lem:kp}, in contradiction with part (i) for the local Gorenstein ring $R_{\p}$.

(iv) Notice that by (i) and (iii) we have
\[ \Tor{k-i+j}R{E(R/\p)}{E_j(M)} = 0 = \Tor{k-i+j+1}R{E(R/\p)}{E_j(M)} \]
for every $0 \leq j < i$, where $E_j(M)$ is the $j$-th term of a minimal injective coresolution of $M$. Indeed, the right hand side equality for $j = i-1$ follows as in (iii) with $\Cosyz{i-1}M$ in place of $M$, together with the assumption that $\mu_{i-1}(\p,M) = 0$.

Now, the short exact sequences $0 \to \Cosyz jM \to E_j(M) \to \Cosyz{j+1}M \to 0$, where $j$ again ranges from $0$ to $i-1$, give rise to exact sequences
\begin{multline*}
0 = \Tor{k-i+j+1}R{E(R/\p)}{E_j(M)} \la \Tor{k-i+j+1}R{E(R/\p)}{\Cosyz{j+1}M} \la \\
\la \Tor{k-i+j}R{E(R/\p)}{\Cosyz jM} \la \Tor{k-i+j}R{E(R/\p)}{E_j(M)} = 0.
\end{multline*}
Thus, $\Tor{k-i+j+1}R{E(R/\p)}{\Cosyz{j+1}M} \cong \Tor{k-i+j}R{E(R/\p)}{\Cosyz jM}$ for each $j < i$, and by induction:
\[ \Tor{k}R{E(R/\p)}{\Cosyz iM} \cong \Tor{k-i}R{E(R/\p)}M. \]
The second claim is an immediate consequence of part (iii) applied to $\Cosyz iM$ and of Lemma~\ref{lem:Bass}.
\end{proof}

A direct consequence is another expression of an $n$-cotilting class over a Gorenstein ring, which is alternative to the ones in Theorem~\ref{thm:class-main} and follows directly from Lemma~\ref{lem:torinj}(iii) and~(iv).

\begin{proposition} \label{prop:another}
Let $R$ be a Gorenstein ring, $(Y_1, \dots, Y_n)$ a sequence of subsets of $\Spec{R}$ as in Definition~\ref{def:cx} and $\C = \C_{(Y_1, \dots, Y_n)}$ the corresponding $n$-cotilting class, following Theorem~\ref{thm:class-n-cot}. Then
\[ \C = \{ M \in \rmod R \mid \Tor{\htt{\p}-i+1}R{E(R/\p)}M = 0 \textrm{ for all } i = 1, \dots, n \textrm{ and } \p \in Y_i \}. \]
\end{proposition}

\begin{proof}
This is obtained merely by combining the description of $\C$ in Definition~\ref{def:cx} with Lemma~\ref{lem:torinj}(iii) and (iv).
\end{proof}

%

Specializing Theorem~\ref{thm:class-main} and Corollary~\ref{cor:optimized} to Gorenstein rings, we almost immediately get a formula as in Proposition~\ref{prop:another}, but with finitely generated modules. Some price must be paid for this, however, in terms of associated prime ideals, as we will see later in Remark~\ref{rem:Lp_ht0}. Recall that as in Corollary~\ref{cor:optimized} we denote for a set $Y \subseteq \Spec R$ by $\bar Y$ the set of all minimal elements of the poset $(Y, \subseteq)$. We also introduce a notation which we will use in the rest of the paper:

\begin{definition} \label{def:Lp}
Let $R$ be Gorenstein and $\p \in \Spec R$ of height $\ge 1$. We denote
\[ L(\p) = \Tr{\Syz{\htt\p-1}{R/\p}}. \]
\end{definition}

\begin{proposition} \label{prop:cotilt_goren}
Let $R$ be a Gorenstein ring, $(Y_1, \dots, Y_n)$ a sequence of subsets of $\Spec{R}$ as in Definition~\ref{def:cx}. Then the $n$-cotilting class $\C = \C_{(Y_1, \dots, Y_n)}$ corresponding to $(Y_1, \dots, Y_n)$ by Theorem~\ref{thm:class-n-cot} equals
\[ \C = \{ M \in \rmod R \mid \Tor{\htt{\p}-i+1}R{L(\p)}M = 0 \textrm{ for all } i = 1, \dots, n \textrm{ and } \p \in \bar Y_i \}. \]
and the associated $n$-tilting class $\T = \T_{(Y_1, \dots, Y_n)}$ equals
\[ \T = \{ M \in \rmod R \mid \Ext{\htt{\p}-i+1}R{L(\p)}M = 0 \textrm{ for all } i = 1, \dots, n \textrm{ and } \p \in \bar Y_i \}. \]
\end{proposition}

\begin{proof}
Given a prime $\p$ of height $k = \htt\p \ge 1$, note that $\Ext iR{R/\p}R = 0$ for all $i = 0, \dots, k-1$. Indeed, this follows from the shape of the injective coresolution of $R$ (see~\cite[Theorem 9.2.27]{EJ}) and the fact that $\Hom R{R/\p}{E(R/\q)} = 0$ for every $\q \in \Spec R \setminus V(\p)$. Thus, $\pd R{L(\p)} = k$ by Lemma~\ref{lem:pdn}(i). Note also that we have for every $i = 1, \dots, k$:
\[ \Syz{k-i}{L(\p)} \cong \Tr{\Syz{i-1}{R/\p}} \]
The statements on $\mathcal C$ and $\mathcal T$ follow from Corollary~\ref{cor:optimized}, using the isomorphisms of functors $\Tor{R}{k - i + 1}{L(\p)}{-} \cong \Tor{R}{1}{\Tr{\Syz{i-1}{R/\p}}}{-}$ and similarly for Ext.
\end{proof}

In connection with Cohen-Macaulay modules and Hochster's conjecture below, we shall be interested in the associated prime ideals of the modules $L(\p)$, or more generally in their Bass invariants. A step toward the goal is to understand what the classes $L^\intercal$ look like for finitely generated modules $L$ of finite flat (hence projective) dimension. Such classes are cotilting class thanks to Lemma~\ref{lem:cofinite}(i), so in particular they are of the form $\C_{(Y_1, \dots, Y_n)}$ for a sequence of subsets of $\Spec R$ as in Definition~\ref{def:cx}. Hence the problem reduces to computing $Y_1, \dots, Y_n$, which for Gorenstein rings amounts to the following general lemma:

\begin{lemma} \label{lem:gor_tor_perp}
Let $R$ be Gorenstein and $L$ be a finitely generated non-projective $R$-module of finite projective dimension $n$. Then $L^{\intercal}$ is an $n$-cotilting class and in view of the correspondence from Theorem~\ref{thm:class-n-cot} we have $L^\intercal = \C_{(Y_1, \dots, Y_n)}$, where
\[ Y_i = \Big\{\p \in \Spec{R} \mid \htt{\p} \geq i \;\mbox{and}\; \p \in \bigcup_{j=0}^{\htt{\p}-i}\Ass{\Cosyz{j}{L}}\Big\} \subseteq \Supp{L} \]
for every $i = 1, \dots, n$. Moreover:
\begin{ii}
 \item For any $0 \leq \ell \leq  n-1$ we have $(\Omega^{\ell}(L))^{\intercal} = \mathcal C_{(Y_{\ell+1}, \dots, Y_n)}$.
 \item For any $\p \in \Spec{R}$ we have $\p \in \Ass{L} \setminus \Ass R$ \iff $\p \in Y_{\htt{\p}}$.
\end{ii}
\end{lemma}

\begin{proof}
We first focus on the properties of the subsets $Y_i \subseteq \Spec R$ as in the statement. The fact that $Y_i \subseteq \Supp{L}$ for any $1\leq i \leq n$ follows easily by Lemma~\ref{lem:Bass}.
We also prove that $(Y_1, \dots, Y_n)$ satisfies the conditions of Definition~\ref{def:cx}. Indeed, the second condition follows directly and the third condition follows from Remark \ref{rem:cx}. It remains to prove that each $Y_i$ is closed under specialization. Let us choose arbitrary $\p \in Y_i$. Since $R$ is noetherian, we only need to prove that $\q \in Y_i$ for minimal prime ideals $\q$ \st $\q \supsetneq \p$. Let us fix such $\q$. Assuming $\p \in Y_i$, we know that $\htt{\p} \geq i$ and there is $0 \leq j \leq \htt{\p}-i$ such that $\mu_{j}(\p, L) \not= 0$. By \cite[Proposition 9.2.13]{EJ}, $\mu_{j+1}(\q,L) \not= 0$ and since $R$ is Cohen-Macaulay by \cite[Theorem 2.1.12]{BrH}, we have $\htt{\q} = \htt{\p}+1$. It follows that $j+1 \leq \htt{\q} - i$ and $\q \in Y_i$.

Next, denote $\D = L^{\intercal}$. As mentioned above, $\D$ is $n$-cotilting by Lemma \ref{lem:cofinite}, where $n = \pd RL$. So is $\C_{(Y_1, \dots, Y_n)}$ by Theorem~\ref{thm:class-n-cot} and the above paragraph. Our task is to prove that the two classes are equal. We will show more. By induction on $i = n,  \dots, 1$, we show that $(\C_{(Y_1, \dots, Y_n)})_{(i)} = \D_{(i)}$ for all $1 \le i \le n$. Note that
\[ \D_{(i)} = \{ M \in \rmod{R} \mid \Tor{j}{R}{L}{M} = 0 \textrm{ for all } j \geq i\} \]
by Remark~\ref{rem:C2} and dimension shifting. In view of Theorem \ref{thm:class-n-cot} and Proposition \ref{prop:k(p)_inside}(i) we need to show that $\p \in Y_i$  \iff the $k(\p)$-filtered module $E(R/\p)$ is \emph{not} contained in $\D_{(i)}$.

Let $i = n$. From \cite[Theorem 2.2]{X} we learn that $\mu_j(\p,L) \ne 0$ only for $\htt{\p}-n \le j \le \htt{\p}$. In particular we have
\[ Y_n = \{ \p \in \Spec R \mid \htt\p \ge n \textrm{ and } \mu_{\htt{\p}-n}(\p,L) \not= 0 \}. \]
Given $\p$ of height $n$, then $\p \in Y_n$ \iff $\Tor{n}{R}{L}{E(R/\p)} \not= 0$ \iff $E(R/\p) \not \in \mathcal D_{(n)}$ by Lemma~\ref{lem:torinj}(iii). For $\p \in Y_n$ of height greater than $n$ we get the same conclusion by Lemma~\ref{lem:torinj}(iii) and~(iv) and by the fact that $\pd RL = n$.

Now suppose that $(\mathcal C_{(Y_1, \dots, Y_n)})_{(i)} = \mathcal D_{(i)}$ for some $2 \leq i \leq n$ and take an arbitrary $\p \in Y_{i-1}$. If $\p \in Y_i$ then $E(R/\p) \not \in \D_{(i)} \supseteq \mathcal D_{(i-1)}$. So we can suppose that $\p \not\in Y_i$, which means that $\htt{\p} < i$ or $\mu_{j}(\p, L) = 0$ for every $0 \leq j \leq \htt{\p} - i$. If $\htt{\p} < i$ then necessarily $\htt{\p} = i-1$ and $\mu_0(\p,L) \not= 0$. Lemma~\ref{lem:torinj}(iii) implies that $\Tor{i-1}{R}{L}{E(R/\p)} \not= 0$ and $E(R/\p) \not \in \D_{(i-1)}$. In the second case, that is, if $\htt{\p} \ge i$ but $\mu_{j}(\p, L) = 0$ for every $0 \leq j \leq \htt{\p} - i$, we must have $\mu_{\htt\p-i+1}(\p, L) \ne 0$. It follows from Lemma~\ref{lem:torinj}(iii) and (iv) that $\Tor{i-1}{R}{L}{E(R/\p)} \not= 0$ again and so $E(R/\p) \not \in \D_{(i-1)}$, too.

Conversely, let $\p$ be \st $E(R/\p) \not\in \D_{(i-1)}$. Clearly $\htt{\p} \geq i-1$ by Lemma~\ref{lem:torinj}(ii). If $E(R/\p) \not \in \D_{(i)}$ then by induction assumption $\p \in Y_i \subseteq Y_{i-1}$, so suppose that $E(R/\p) \in \D_{(i)}$. This means that $\Tor{j}{R}{L}{E(R/\p)} = 0$ for $j \geq i$. By Lemma \ref{lem:torinj}(iii) and (iv) it follows that $\mu_{\htt{\p}-(i-1)}(\p,L) \not=0$ and  $\p \in Y_{i-1}$.

Now we prove the two additional statements. Part (i) follows from Corollary~\ref{cor:cx_shifting} and Remark~\ref{rem:C2}. Part (ii) follows easily from the description of the sets $Y_i$ above.

\end{proof}

Now, given $\p$ of height at least one, it is rather easy to show that  $\p$ is the only prime ideal of positive height  in $\Ass L(\p)$.


\begin{proposition} \label{prop:ass_Lp}
Let $R$ be a Gorenstein ring, $\p$ be a prime ideal of height $k\ge 1$, and let $L(\p) = \Tr{\Syz{k-1}{R/\p}}$ as in Definition~\ref{def:Lp}. Then $\pd{R}{L(\p)} = k$ and $\Ass{L(\p)} \setminus \Ass{R} = \{\p\}$.
%
\end{proposition}

\begin{proof}
The fact that $\pd{R}{L(\p)} = k$ has been shown in the proof of Proposition~\ref{prop:cotilt_goren}. By the same proposition, $L(\p)^{\intercal} = \mathcal C_{(Y_1, \dots, Y_k)}$ for $Y_1 = \dots = Y_k = V(\p)$. The statement then follows from Lemma \ref{lem:gor_tor_perp}(ii).
\end{proof}

\begin{remark} \label{rem:Lp_ht0}
Although $\Ass E(R/\p) = \{\p\}$ and Propositions~\ref{prop:another} and~\ref{prop:cotilt_goren} are formally rather similar, one cannot easily get rid of the zero height prime ideals in $\Ass L(\p)$. First of all, $L(\p)$ is only defined uniquely up to adding or splitting off a projective summand; recall Definition~\ref{def:Tr} and the comment below it.

There is a more substantial problem, however. If $\Ass L(\p) = \{\p\}$ for a particular choice of $L(\p)$, then we have $\Hom R{L(\p)}R = 0$ since $\Supp L(\p) \cap \Ass R = \emptyset$. This would imply that $\pd R{R/\p} \le \htt\p$ by the very construction of $L(\p)$. As far as we are concerned, this is a trivial situation.
In that case, we could replace $L(\p)$ by $R/\p$ in the formula in Proposition~5.4, as we will see below in Theorem~\ref{thm:equicm}. In fact, $R/\p$ would then be a Cohen-Macaulay module by Lemma~\ref{lem:relation}. The latter is certainly not true in general.
\end{remark}

\subsection{Cohen-Macaulay modules and Hochster's conjecture}
\label{subsec:hochster}

In Propositions \ref{prop:another} and \ref{prop:cotilt_goren} we get two different expressions of cotilting classes over Gorenstein rings. Now we are going to discuss the possibility of combining these two attempts. Namely we would like to find a finitely generated module $K(\p)$ for each $\p \in \Spec{R} \setminus \Ass{R}$ \st $\pd{R}{K(\p)} = \htt{\p}$, $\Ass{K(\p)} = \{\p\}$ and such that these modules can be used to express any cotilting class. We will see later that the last property follows from the other two and that this attempt leads to the question of existence of some Cohen-Macaulay modules. Let us recall some relevant definitions and results.

\begin{definition} \label{def:cm}
Let $R$ be a commutative noetherian local ring and $M \in \rfmod R$. Then $M$ is \emph{Cohen-Macaulay} if $M \neq 0$ and $\dt M = \Kd M$, where $\dt M$ denotes the depth of $M$ and $\Kd M$ the Krull dimension of~$M$. 
A Cohen-Macaulay module is called \emph{maximal Cohen-Macaulay}, if moreover $\dt M = \Kd R$. 
If $M$ is maximal Cohen-Macaulay, then the Auslander-Buchsbaum formula \cite[9.2.20]{EJ} implies that either $M$ has infinite projective dimension, or else $M$ is free. 

If $R$ is a general commutative noetherian ring and $M \in \rfmod R$, then $M$ is \emph{Cohen-Macaulay} if $M_{\m}$ is a Cohen-Macaulay $R_{\m}$-module for each maximal ideal $\m \in \Supp M$. The ring $R$ is called \emph{Cohen-Macaulay} if it is Cohen-Macaulay as a module over itself.
\end{definition}

\begin{lemma} \label{lem:translate}
Let $R$ be a Gorenstein ring, $\p \in \Spec R$, and $K \in \rfmod R$ be such that $\Ass K = \{ \p \}$. Then the following are equivalent:
\begin{ii}
\item $K$ is a Cohen-Macaulay module such that $\pd RK < \infty$;
\item $\pd RK = \htt \p$.
\end{ii}
\end{lemma}

\begin{proof}
If (i) holds, then for each maximal ideal $\m \in \Supp K$, $K_{\m}$ is a Cohen-Macaulay $R_{\m}$-module of finite projective dimension, and the Auslander-Buchsbaum formula \cite[9.2.20]{EJ} gives $\pd {R_{\m}}{K_{\m}} = \htt \m - \Kd {K_{\m}}$. Since $\Ass K = \{ \p \}$,  we get $\Kd {K_{\m}} = \Kd (R/\p)_{\m} = \htt \m - \htt \p$. This proves that $\pd RK = \htt \p$.

Conversely, if (ii) holds then for each maximal ideal $\m \in \Supp K$, we have $\dt {K_{\m}} = \Kd R_{\m} - \pd {R_{\m}}{K_{\m}} \geq \Kd R_{\m} - \htt \p = \Kd (R/\p)_{\m} = \Kd K_{\m}$, so  $K_{\m}$ is a Cohen-Macaulay module.
\end{proof}

Now we shall show how to express any cotilting class using Cohen-Macaulay modules as in the latter lemma. Using the convention of Corollary~\ref{cor:optimized}, given a set $Y \subseteq \Spec R$, we denote by $\bar Y$ the set of all minimal elements of the poset $(Y, \subseteq)$.

\begin{theorem} \label{thm:equicm}
\begin{sloppypar}
Let $R$ be a Gorenstein ring and assume that for each $\p \in \Spec R \setminus \Ass R$ there exists a Cohen-Macaulay module $K(\p) \in \rfmod R$ such that $\pd R{K(\p)} = \htt{\p}$ and $\Ass K(\p) = \{ \p \}$. Then for each $(Y_1, \dots, Y_n)$ as in Definition~\ref{def:cx}, the $n$-tilting class corresponding to $(Y_1, \dots, Y_n)$ by Theorem~\ref{thm:class-main} equals
\[
\T_{(Y_1, \dots, Y_n)} =
\{ M \mid \Ext{\htt{\p}-i+1}{R}{K(\p)}{M} = 0 \textrm{ for } i = 1, \dots, n \textrm{ and } \p \in \bar{Y_i} \}
\eqno{\phantom{*}(*)}
\]
and the $n$-cotilting class corresponding to $(Y_1, \dots, Y_n)$ by Theorem \ref{thm:class-main} is 
\[
\C_{(Y_1, \dots, Y_n)} =
\{ M \mid \Tor{\htt{\p}-i+1}R{K(\p)}M = 0 \textrm{ for } i = 1, \dots, n \textrm{ and } \p \in \bar{Y_i} \}
\eqno{(**)}
\]
\end{sloppypar}
\end{theorem}

\begin{proof}
In view of Lemma \ref{lem:cofinite}, it suffices to prove the assertion concerning $\C_{(Y_1, \dots, Y_n)}$.

We first claim that for any $\p \in \Spec{R} \setminus \Ass{R}$ we have $K(\p)^{\intercal} = \C_{(Y^\p_1, \dots, Y^\p_{\htt\p})}$ where we denote $Y^\p_1 = \dots = Y^\p_{\htt{\p}} = V(\p)$. Indeed, $K(\p)^{\intercal}$ is a cotilting class, hence of the form $K(\p)^{\intercal} = \C_{(Y'_1, \dots, Y'_{\htt{\p}})}$. By Lemma \ref{lem:gor_tor_perp} we have $Y'_i \subseteq \Supp K(\p)=V(\p)$ and by~Definition \ref{def:cx} we know that $Y'_i$ is closed under specialization and $Y'_i \supseteq Y'_{\htt\p}$ for any $1 \leq i \leq \htt{\p}$. So it is enough to prove that $\p \in Y'_{\htt\p}$, but this has been shown in Lemma~\ref{lem:gor_tor_perp}(ii).

Having proved the claim, we prove the equality for $\C_{(Y_1, \dots, Y_n)}$ by induction of $n$. Let us denote the class $\{ M \mid \Tor{\htt{\p}-i+1}R{K(\p)}M = 0 \textrm{ for } i = 1, \dots, n \textrm{ and } \p \in \bar{Y_i} \}$ from the statement by $\D$. If $n=1$, then, using that $\pd R{K(\p)} = \htt{\p}$, we have $\D = \bigcap_{\p \in \bar{Y}_1} \Syz{\htt\p-1}{K(\p)}^\intercal$. Since $\Syz{\htt\p-1}{K(\p)}^\intercal = \C_{(Y^\p_{\htt\p})} = \C_{(V(\p))}$ by the claim and Lemma~\ref{lem:gor_tor_perp}(i), it easily follows that $\C_{(Y_1)} = \D$.

If $n>1$, we infer from the inductive hypothesis and the fact that $\bar Y_n \subseteq Y_i$ for all $i = 1, \dots, n-1$ that
\[ \D = \C_{(Y_1, \dots, Y_{n-1})} \cap \bigcap_{\p \in \bar{Y}_n} \Syz{\htt\p-n}{K(\p)}^\intercal. \]
Invoking the claim and Lemma~\ref{lem:gor_tor_perp}(i) again, we obtain
\[
\D =
\C_{(Y_1, \dots, Y_{n-1})} \cap \bigcap_{\p \in \bar{Y}_n} \C_{(Y^\p_{\htt{\p}-n+1}, \dots, Y^\p_{\htt\p})} =
\C_{(Y_1, \dots, Y_n)}.
\]
The last equality is an easy consequence of Definition~\ref{def:cx}.
\end{proof}

The existence of Cohen-Macaulay modules as in Theorem~\ref{thm:equicm} is in general only conjectural even for regular rings: a ring $R$ is \emph{regular} in case $R$ is commutative noetherian and the ring $R_{\p}$ has finite global dimension for each $\p \in \Spec R$. By \cite[Corollary 2.2.20]{BrH}, $R \cong \prod_{\p \in \Ass{R}} R/{\p}$ where $R/{\p}$ is a regular domain.

From now on, we will restrict ourselves to regular rings. Let us now complete the argument from Remark~\ref{rem:Lp_ht0} to put our results into a proper context.

\begin{lemma} \label{lem:relation}
Let $R$ be a regular ring and $\p \in \Spec R$. Then 
\[ \htt \p = \pd  {R_{\p}}{k(\p)} \leq \pd R{R/\p}. \] 
The equality $\htt \p = \pd R{R/\p}$ holds \iff $R/\p$ is a Cohen-Macaulay ring; in this case $\pd  {R_{\q}}{(R/\p)_{\q}} = \pd R{R/\p}$ for all $\q \in \Spec R$ with $\p \subseteq \q$. 
\end{lemma}

\begin{proof}
First, $k(\p) = (R/\p)_{\p}$, so $\dt (R/\p)_{\p} = 0$ (cf.\ \cite[9.2.9]{EJ}). Using the Auslander-Buchsbaum formula \cite[9.2.20]{EJ}, we infer that $\pd {R_{\p}}{k(\p)} = \htt \p$. If $\p \subsetneq \q$, then $\dt R_{\q}/\p_{\q} \leq \htt \q - \htt \p$, so $\pd {R_{\q}}{R_{\q}/\p_{\q}} = \htt \q - \dt R_{\q}/\p_{\q} \geq \htt \p$. 

Since $\pd R{R/\p} = \max _{\q \in \Spec R} \pd {R_{\q}}{(R/\p)_{\q}}$, clearly $\pd R{R/\p} \geq \htt \p$. However, $R/\p$ is Cohen-Macaulay, iff the equality $\dt R_{\q}/\p_{\q} = \Kd R_{\q}/\p_{\q}$ holds for all $\q \in \Spec R$ with $\p \subseteq \q$. As $\Kd R_{\q}/\p_{\q} = \Kd R_{\q} - \htt \p = \htt \q - \htt \p$ by \cite[2.1.4]{BrH}, we may apply the Auslander-Buchsbaum formula again to conclude that $R/\p$ is Cohen-Macaulay, iff $\pd {R_{\q}}{(R/\p)_{\q}} = \htt \p$ for all $\q \in \Spec R$ with $\p \subseteq \q$. 
\end{proof}  

The assumptions of Theorem~\ref{thm:equicm} are always met for regular rings of Krull dimension $\leq 3$. In the context of tilting and cotilting classes, this leads to a simplification of the formulas from Proposition~\ref{prop:cotilt_goren}:

\begin{corollary} \label{cor:3}
If $R$ is a regular ring with $\Kd R \leq 3$, then $R/\p$ is Cohen-Macaulay for each $\p \in \Spec R$ of height $\geq 1$. In particular, given a sequence $(Y_1,Y_2,Y_3)$ as in Definition~\ref{def:cx}, the tilting class corresponding to this sequence by Theorem \ref{thm:class-main} equals 
\[ \T_{(Y_1,Y_2,Y_3)} = \{ M \mid \Ext{\htt{\p}-i+1}{R}{R/\p}{M} = 0 \textrm{ for } i = 1, 2, 3 \textrm{ and } \p \in \bar{Y_i} \} \]
and the cotilting class corresponding to it by Theorem \ref{thm:class-main} is 
\[ \C_{(Y_1,Y_2,Y_3)} = \{ M \mid \Tor{\htt{\p}-i+1}{R}{R/\p}{M} = 0 \textrm{ for } i = 1, 2, 3 \textrm{ and } \p \in \bar{Y_i} \} \]
\end{corollary}

\begin{proof} We can w.l.o.g.\ assume that $R$ is a regular domain. We must then prove that $R/\p$ is Cohen-Macaulay for each $0 \neq \p \in \Spec R$. This is trivial when $\p$ has height $3$. The cases of height $1$ and $2$ are proved by localization: if $\p$ has height $2$, then the localization of $R/\p$ at any maximal ideal is a $1$-dimensional local domain which is necessarily Cohen-Macaulay \cite[p.64]{BrH}. Finally, each regular local ring is a UFD, so its prime ideals of height $1$ are principal, hence $R/\p$ is even Gorenstein for $\p$ of height $1$, see \cite[3.1.19(b)]{BrH}. 
\end{proof}

However, the existence of Cohen-Macaulay modules $K(\p)$ as in Lemma~\ref{lem:translate} in broader generality is closely related to long standing open problems in commutative algebra. One of them is:

\begin{conjHochster} \cite[Conjecture (E)]{H} 
Each complete local ring possesses a maximal Cohen-Macaulay module.
\end{conjHochster}

Since factors of complete local rings are complete, and each complete local ring is a factor of a complete regular local ring, Hochster's Conjecture can equivalently be stated as follows: for each complete regular local ring $R$ and each $\p \in \Spec R$ there exists a maximal Cohen-Macaulay $R/\p$-module $K(\p)$. In \cite[\S3]{H}, Hochster's Conjecture is proved for rings of Krull dimension $\leq 2$. In fact, the canonical $R/\p$-module $K(\p) = \Ext 2R{R/\p}R$ satisfies $\dt K(\p) = \Kd K(\p) = \Kd R/\p = 2$, so $K(\p)$ is a maximal Cohen-Macaulay $R/\p$-module in that case (see \cite[Example 3.2(b)]{Sc}). In general, however, the conjecture remains wide open. 

The following lemma shows that in the complete local case, Hochster's Conjecture implies the existence of Cohen-Macaulay modules as in Lemma~\ref{lem:translate} for each $\p \in \Spec R$:

\begin{lemma} \label{lem:mcm}
Let $R$ be a regular local ring and $\p \in \Spec R$. Assume there exists a maximal Cohen-Macaulay $R/\p$-module $K(\p)$. Then viewed as an $R$-module, $K(\p)$ is Cohen-Macaulay and satisfies $\Ass K(\p) = \{ \p \}$.
\end{lemma}

\begin{proof}
The maximality of $K(\p)$ implies that $K(\p)$ is a torsion-free $R/\p$-module by \cite[21.9]{E}. So $K(\p) \subseteq (R/\p)^n$ for some $n < \omega$ by \cite[Proposition VII.2.4]{CE}. Considered as an $R$-module, $K(\p)$ thus satisfies $\Ass K(\p) = \{ \p \}$ which implies that $K(\p)$ is a Cohen-Macaulay $R$-module.
\end{proof}

To see how limiting the assumption of existence of modules $K(\p)$ from Theorem~\ref{thm:equicm} is, we relate it to Serre's Positivity conjecture. In order to state it, we recall the notion of the intersection multiplicity:

\begin{definition}\label{def:im} Let $R$ be a regular local ring of Krull dimension $d$ and $M, N \in \rfmod R$ be such that $M \otimes_R N$ has finite length. Then the \emph{intersection multiplicity} of $M$ and $N$ is defined as
$$
 \chi(M,N) = \sum_{i = 0}^{d}(-1)^i\,\mbox{length}\,(\Tor{i}{R}{M}{N}).
$$
\end{definition}

\begin{conjSerre} \cite{JPS2000}
Assume that $R$ is a regular local ring of Krull dimension~$d$, and $M, N \in \rfmod R$ are such that $M \otimes_R N$ has finite length. Then

\begin{enumerate}
 \item $\Kd M + \Kd N \leq \Kd R$;
 \item \emph{({\bf Vanishing})} $\Kd M + \Kd N < \Kd R \;\Longrightarrow\; \chi(M,N) = 0$;
 \item \emph{({\bf Positivity})} $\Kd M + \Kd N = \Kd R \;\Longrightarrow\; \chi(M,N) > 0$.
\end{enumerate}
\end{conjSerre}

Serre proved (1) in general, and he also proved (2) and (3) for all regular local rings containing a field. The Vanishing Conjecture was proved by Roberts \cite{PR1985}. The Positivity Conjecture was proved for rings of Krull dimension $\leq 4$ by Hochster \cite{H}, but it is still open in general (However, Gabber proved that $\chi(M,N) \geq 0$).

\begin{lemma} \label{lem:fromH}
Let $R$ be a regular local ring and assume that for each $\p \in \Spec{R}$ there exists $K(\p) \in \rfmod R$ such that 
$\Ass K(\p) = \{\p\}$ and $K(\p)$ is a Cohen-Macaulay module. Then Serre's Positivity Conjecture holds for $R$.
\end{lemma}

\begin{proof}
The idea of the proof is taken from \cite[Theorem 2.9]{H}. By \cite[V. B4, Remark c)]{JPS2000} it is enough to prove the conjecture for $M = R/\p$ and $N = R/\q$ such that $R/\p \otimes_R R/\q$ has finite length. So take such $\p, \q \in \Spec{R}$, and note that a finitely generated module $L$ has finite length $>0$ \iff $\Supp{L} = \{\m\}$,  where $\m$ is the unique maximal ideal of $R$.

Let us now compute $\chi(K(\p), K(\q))$. By \cite[Proposition 9.2.7]{NIEL}, $\Supp{K(\p) \otimes_R K(\q)} = \Supp {K(\p)} \cap \Supp{K(\q)}$. It follows that $K(\p) \otimes_R K(\q)$ has finite length by the note above. By \cite[V. B6, Corollary to Theorem 4]{JPS2000},  $\Tor{i}{R}{K(\p)}{K(\q)} = 0$ for $i > 0$. Thus $\chi(K(\p), K(\q)) = \,\mbox{length}\,(K(\p) \otimes_R K(\q)) > 0$. 

But $K(\p)$ can be filtered by $\{R/\p' \mid \p' \in V(\p)\}$ and similarly for $K(\q)$. Since $\chi(-,-)$ is additive on short exact sequence in both variables and since the vanishing conjecture holds for $R$ (\cite{PR1985}), it follows that $\chi(R/\p, R/\q) > 0$.
\end{proof}

We finish by discussing the relation of our modules $L(\p)$ from Definition~\ref{def:Lp} to the finitely generated maximal Cohen-Macaulay $R/\p$-modules $K(\p)$ whose existence has been conjectured by Hochster. We 
know from Proposition~\ref{prop:ass_Lp} and Lemma ~\ref{lem:translate}
that for $\p$ of height $\ge 1$, the module $L(\p)$ from Definition~\ref{def:Lp} is Cohen-Macaulay if it has no associated primes of height $0$. On the other hand, Remark~\ref{rem:Lp_ht0} and Lemma~\ref{lem:relation} tell us that in such a case we could as well take $K(\p) = R/\p$ since it is also Cohen-Macaulay. Hence, the statement of our final result is necessarily more elaborate.

\begin{theorem} \label{thm:Lp_to_Kp}
Let $R$ be a regular local ring, $\p$ be a prime ideal of non-zero height, and denote $L(\p) = \Tr{\Syz{\htt\p-1}{R/\p}}$. Assume there exists a maximal Cohen-Macaulay $R/\p$-module $K(\p) \in \rfmod{R/\p}$.

Then for each $i = 0, \dots, \htt \p - 1$, the smallest resolving class $\mathcal L(\p) \subseteq \rfmod R$ containing $\Syz{i}{L(\p)}$ coincides with the smallest resolving class $\mathcal K(\p) \subseteq \rfmod R$ containing $\Syz{i}{K(\p)}$. In particular, any finitely generated $i$-th syzygy of the $R$-module $K(\p)$ is a direct summand in a finitely $\mathcal E$-filtered module where
\[
\mathcal E = \{ \Syz{j}{L(\p)} \mid j = i, \dots , \htt \p - 1 \} \cup \{ R \}.
\]
\end{theorem}

\begin{proof}
Note that $\p\not\in\Ass R$. Let us fix $i \in \{0, \dots, \htt{\p} -1\}$ and denote $L = \Syz{i}{L(\p)}$ and $K = \Syz{i}{K(\p)}$. Note that $L = \Tr{\Syz{\htt\p-i-1}{R/\p}}$, as observed in the proof of Proposition~\ref{prop:cotilt_goren}. Using Corollary~\ref{cor:optimized} and Theorem~\ref{thm:equicm} we infer that
\[ L^\perp = K^\perp = \T_{(Y_1, \dots, Y_{\htt\p})} \]
for $Y_1 = \dots = Y_{\htt \p -i} = V(\p)$ and $Y_{\htt \p -i +1} = \dots = Y_{\htt \p} = \emptyset$. The claim follows by Corollary~\ref{cor:parres} and its proof.
\end{proof}

\section*{Acknowledgment}

Our thanks for valuable advice are due to Hailong Dao (concerning relations between Hochster's and Serre's conjectures) and Santiago Zarzuela (on Cohen-Macaulay and canonical modules).

\bibliographystyle{abbrv}
\bibliography{ncotilting+}

\end{document}